\pdfoutput=1
\RequirePackage{amsmath}
\RequirePackage{fix-cm}
\RequirePackage{rotating}

\documentclass[11pt]{scrartcl}
 
\usepackage[utf8x]{inputenc}
\usepackage[T1]{fontenc}
\usepackage[english]{babel}
\usepackage{lmodern}

\usepackage{amsmath,amssymb,amstext,mathtools}
\usepackage{amsthm}

\usepackage{todonotes}

\usepackage[normalem]{ulem}

\usepackage{pdflscape}

\usepackage{hyperref}
\usepackage{cleveref} 
\usepackage{url}

\usepackage{algorithm}
\usepackage{algpseudocode}

\usepackage{longtable,booktabs}
\usepackage{csvsimple,pgfkeys}

\newtheorem{theorem}{Theorem}
\newtheorem{proposition}[theorem]{Proposition}
\newtheorem{corollary}[theorem]{Corollary}
\newtheorem{lemma}[theorem]{Lemma}
\newtheorem{remark}[theorem]{Remark}
\newtheorem*{definition*}{Defintion}

\usepackage{scalerel}
\usepackage{tikz}
\usetikzlibrary{svg.path}

\definecolor{orcidlogocol}{HTML}{A6CE39}
\tikzset{
	orcidlogo/.pic={
		\fill[orcidlogocol] svg{M256,128c0,70.7-57.3,128-128,128C57.3,256,0,198.7,0,128C0,57.3,57.3,0,128,0C198.7,0,256,57.3,256,128z};
		\fill[white] svg{M86.3,186.2H70.9V79.1h15.4v48.4V186.2z}
		svg{M108.9,79.1h41.6c39.6,0,57,28.3,57,53.6c0,27.5-21.5,53.6-56.8,53.6h-41.8V79.1z M124.3,172.4h24.5c34.9,0,42.9-26.5,42.9-39.7c0-21.5-13.7-39.7-43.7-39.7h-23.7V172.4z}
		svg{M88.7,56.8c0,5.5-4.5,10.1-10.1,10.1c-5.6,0-10.1-4.6-10.1-10.1c0-5.6,4.5-10.1,10.1-10.1C84.2,46.7,88.7,51.3,88.7,56.8z};
	}
}
\newcommand\orcidicon[1]{\href{https://orcid.org/#1}{\mbox{\scalerel*{
				\begin{tikzpicture}[yscale=-1,transform shape]
				\pic{orcidlogo};
				\end{tikzpicture}
			}{|}}}}

\DeclareMathOperator*{\argmin}{\arg\,min\,}

\DeclarePairedDelimiter{\set}{\{}{\}}
\DeclarePairedDelimiter{\roundbr}{(}{)}
\DeclarePairedDelimiter{\floor}{\lfloor}{\rfloor}
\DeclarePairedDelimiter{\inprod}{\langle}{\rangle}
\DeclarePairedDelimiter{\norm}{\lVert}{\rVert}

\newcommand{\R}{\mathbb{R}}
\newcommand{\Z}{\mathbb{Z}}
\newcommand{\conv}{\operatorname{conv}}
\newcommand{\aff}{\operatorname{aff}}
\newcommand{\tr}{\operatorname{tr}}
\newcommand{\diag}{\operatorname{diag}}
\newcommand{\Diag}{\operatorname{Diag}}
\newcommand{\xs}{\bar x}
\newcommand{\zs}{\bar z}
\newcommand{\ys}{\bar y}
\newcommand{\Ys}{\bar Y}
\newcommand{\xm}{\widetilde x}
\newcommand{\ym}{\widetilde y}
\newcommand{\Xm}{\widetilde X}
\newcommand{\Ym}{\widetilde Y}
\newcommand{\Cm}{\widetilde C}
\newcommand{\dm}{\widetilde d}
\newcommand{\Lm}{\widetilde L}
\newcommand{\laplacian}{L}
\newcommand{\Lb}{L}
\newcommand{\Cb}{C}
\newcommand{\db}{d}
\newcommand{\xb}{x}

\newcommand{\yb}{y}
\newcommand{\Yb}{Y}
\newcommand{\Yxl}{\widetilde Y}
\newcommand{\Lxl}{\widetilde L}
\newcommand{\symmmat}[1]{\mathcal{S}^{#1}}
\newcommand{\projpsd}{\mathcal{P}_{\succeq 0}}
\newcommand{\projnsd}{\mathcal{P}_{\preceq 0}}
\newcommand{\dualvarone}{\nu}
\newcommand{\dualvartwo}{\mu}

\title{Connectivity via convexity: Bounds on the edge expansion in graphs}

\author{
Timotej Hrga\footnote{Institut f\"ur Mathematik,
  Alpen-Adria-Universit\"at Klagenfurt, Universit\"atstra{\ss}e 65-67,
  9020 Klagenfurt,
  \href{mailto:timotej.hrga@aau.at}{timotej.hrga@aau.at}, \href{mailto:melanie.siebenhofer@aau.at}{melanie.siebenhofer@aau.at}, \href{mailto:angelika.wiegele@aau.at}{angelika.wiegele@aau.at}}
\footnote{Corresponding Author: \href{mailto:timotej.hrga@aau.at}{timotej.hrga@aau.at}}  
~\orcidicon{0000-0002-4852-1986} 
\and Melanie Siebenhofer$^{*}$\orcidicon{0000-0002-9101-834X}
\and Angelika Wiegele$^{*}$\orcidicon{0000-0003-1670-7951}
}

 \date{\today{}}

\begin{document}

\maketitle

\begin{abstract}
  Convexification techniques have gained increasing interest over 
  the past decades. In this work, we apply a recently developed
  convexification technique for fractional programs by He, Liu and
  Tawarmalani~(2024) to the problem of determining the edge expansion
  of a graph.
  Computing the edge expansion of a graph is a well-known, difficult combinatorial
  problem that seeks to partition the graph into two sets such that a
  fractional objective function is minimized.
  
  We give a formulation of the edge expansion as a completely positive
  program and propose a relaxation as a doubly non-negative program,
  further strengthened by cutting planes.
  Additionally, we develop an augmented Lagrangian algorithm to solve the
  doubly non-negative program, obtaining lower bounds on the
  edge expansion.
  Numerical results confirm that this relaxation yields strong bounds
  and is computationally efficient, even for graphs with several hundred
  vertices. 
  \bigskip
  
\end{abstract}

\paragraph{Acknowledgements}
This research was funded in part by the Austrian Science Fund (FWF) [10.55776/DOC78]. For open access purposes, the authors have applied a CC BY public copyright license to any author-accepted manuscript version arising from this submission.
This research was initiated through discussions with Mohit Tawarmalani at the workshop “Mixed-integer Nonlinear Optimization: a hatchery for modern mathematics” held at Mathematisches Forschungsinstitut Oberwolfach (MFO) in 2023.

\section{Introduction}\label{sec:intro}

Let $G = (V,E)$ be a simple graph on $n \ge 3$ vertices.
The edge expansion of $G$ is defined as
\begin{equation*}
  h(G) = \min_{\emptyset \neq S \subset V}
  \frac{\lvert \partial S \rvert}%
  {\min \{ \lvert S \rvert, \lvert V \setminus S \rvert\}},
\end{equation*}
where $\partial S = \big\{ \{i,j\}  \in E : i \in S, j \notin S \big\}$
is the cut induced by the set $S$.
This graph parameter is also known under the name \emph{Cheeger constant}
or \emph{isoperimetric number} or \emph{sparsest cut}.

This constant is positive if and only if the graph is connected.
A graph with $h(G) \ge c$, for some constant $c > 0$,
is called a $c$-expander.
A graph with $h(G) < 1$ is said to have a bottleneck since
there are not too many edges across it.
The famous conjecture of
Mihail-Vazirani~\cite{mihail1992expansion,feder92matroids} in
polyhedral combinatorics claims that the graph (1-skeleton) of any 
0/1-polytope has edge expansion at least~1.
This has been proven to be true for several combinatorial 
polytopes~\cite{mihail1992expansion,kaibel2004expansion} and bases-exchange
graphs of matroids~\cite{Anari2019}, and a weaker form was established
recently for random 0/1-polytopes~\cite{leroux2023randompolytopes}.  
The edge expansion problem arises in several applications. For
references and for relations to similar problems, we refer the reader
to the recent paper~\cite{gupte2024journal}.

The edge expansion problem belongs to the class of combinatorial fractional programming
problems.
Fractional programming has been studied at least since the sixties of
the $20^\text{th}$ century~\cite{dinkelbach1967nonlinear,charnes1962fractional}. 
It finds applications in fields like economics, engineering, and
telecommunications.
While there are fractional programs that can be solved in polynomial
time~\cite{charnes1962fractional}, the edge expansion problem and many
other fractional programs do not admit polynomial time algorithms. For
more details on applications and on known results of fractional
programming, we refer to the comprehensive literature review
in~\cite{he2024convexification}.  

In~\cite{gupte2024journal} two exact algorithms based on semidefinite
programming to compute $h(G)$ have been developed.
In the first algorithm,
the problem is split into several $k$-bisection problems. Limiting the
number of candidates to be considered for $k$ as well as computing the
$k$-bisection is done using strong bounds from semidefinite
programming.
The other algorithm in~\cite{gupte2024journal} uses an idea of Dinkelbach
and computes the edge expansion by considering several parametrized
optimization problems, which can all be solved using a max-cut solver.

We now take a different direction, and instead of transforming the
problem into several instances of known combinatorial optimization
problems, we want to obtain strong lower bounds 
on the edge expansion itself by using a doubly non-negative relaxation. 
This relaxation results from applying a convexification technique for
fractional programs, developed recently by He, Liu and
Tawarmalani~\cite{he2024convexification}.

The main contributions of this paper are as follows.
\begin{itemize}
\item We convexify the edge expansion problem and write it as a completely positive program.
\item We relax this completely positive program to a doubly non-negative program and strengthen it with additional inequalities.
\item We develop and implement an augmented Lagrangian algorithm
  tailored for the derived relaxation.
\item The quality of the bounds produced by the relaxation and the
  effectiveness of our algorithm are illustrated by numerical
  results. 
\item All code and data sets are publicly available as open source.
\end{itemize}

\paragraph{Outline}
This paper is structured as follows. We next formulate the problem as a binary fractional optimization problem and derive a basic doubly non-negative relaxation.
In Section~\ref{sec:theorems-fromhe2024}
we state the main result of~\cite{he2024convexification} that we apply in
Section~\ref{sec:convexification} to give a formulation of the edge
expansion as a completely positive program.
In Section~\ref{sec:dnn}, we then relax the completely positive program
to a doubly non-negative programming problem, propose a facial reduction approach,
and prove that the new bound is at least as good as the one from the basic relaxation. 
An augmented Lagrangian algorithm for solving the relaxation is
presented in Section~\ref{sec:auglag}.
Section~\ref{sec:numericalresults} shows that, indeed, the relaxation
is improving over the basic relaxation significantly
and demonstrates the efficiency of the bounding routine. We conclude the paper in Section~\ref{sec:conclusion}. 

\paragraph{Notation}
The set of $n\times n$ real symmetric matrices is denoted by $\symmmat{n}$.
The positive semidefiniteness condition for $X \in \symmmat{n}$ is written as $X \succeq 0$,
and we denote by~$\projpsd(X)$ the projection of~$X$ onto the cone of positive semidefinite matrices.
We write $X \ge 0$ to denote that a matrix $X$ is entry wise non-negative and $X \in \text{DNN}$ if $X \succeq 0$ and $X \ge 0$.
The trace of $X$ is written as $\tr(X)$ and defined as the sum of its diagonal elements. 
The trace inner product for~$ X,Y \in \symmmat{n}$ is defined as
$\inprod{X, Y} = \tr (XY)$, and we denote the Frobenius norm of $X$ by $\norm{X} = \sqrt{\inprod{X,X}}$.
The operator $\diag(X)$ returns the main diagonal of matrix
$X$ as a vector.
The operator $\Diag \colon \R^n \to \R^{n\times n}$ yields the diagonal matrix with its diagonal
equal to the vector and
$\Diag(M_1,M_2,\dots,M_k)$ returns the block-diagonal matrix constructed by the input matrices
and by the diagonal matrices formed from all the arguments, which are vectors.
The symbol $\otimes$ denotes the Kronecker product.
The vector of all ones of size $n$ is denoted by $e_n$, 
the vector of all zeros of size $n$ is denoted by $0_n$, and $I_n$ is
the $n \times n$ identity matrix. By~$E_n$, we denote the matrix
of all ones, and~$u_i$ is the standard unit vector with entry~$i$ equal to~1.
All the subscripts are omitted if the size is clear from the context.

\section{A basic DNN relaxation for edge expansion}\label{sec:dnn-basic}
The edge expansion problem can be formulated as the binary fractional
optimization problem
\begin{equation}
	\label{eq:fracprobl-edgeexp}
\begin{aligned}
h(G) = \min        \quad & \frac{\xs^\top \laplacian \xs}{e^\top \xs} \\
\text{s.t.} \quad & 1 \leq e^\top \xs \leq \Big\lfloor \frac n 2 \Big\rfloor\\
				  & \xs \in \{0,1\}^n,
\end{aligned}
\end{equation}
where~$\laplacian$ denotes the Laplacian matrix of~$G$ that is defined
as 
\begin{equation*}
	\laplacian_{ij} = 
	\begin{cases}
		-1, 		& \text{if } \{i,j\} \in E,\\
		\deg(i), & \text{if } i = j \text{ and}\\
		0,		& \text{otherwise.}
	\end{cases}
\end{equation*}
To obtain a doubly non-negative relaxation (DNN) for \eqref{eq:fracprobl-edgeexp}, one can linearize the objective function by introducing a matrix variable, which results in a matrix lifting relaxation.
Let $X = \xs \xs^\top$ and $\rho = \frac{1}{e^\top \xs}$. Then the matrix $\Ys = \rho X$ and the vector $\ys = \rho \xs$ satisfy the following conditions. First, we have
\[e^\top\ys = e^\top(\rho \xs) = \frac{e^\top \xs}{e^\top \xs} = 1.\]
Then, because of the constraints $1 \leq e^\top \xs \leq \lfloor \frac n 2 \rfloor$, it follows that \[\frac{1}{\lfloor \frac n 2 \rfloor} \leq \rho \leq 1.\]
From $\langle E, \Ys \rangle = \tr\left(ee^\top(\rho \xs\xs^\top)\right) = \frac{(e^\top \xs)^2}{e^\top \xs}  = e^\top \xs$ we get \[1 \leq \langle E, \Ys \rangle \leq \Big\lfloor \frac n 2 \Big\rfloor.\]
For $\xs_i \in \{0,1\}$, we have $\xs_i^2 = \xs_i$ and hence, \[\diag(\Ys) = \diag(\rho\xs \xs^\top) = \rho \xs = \ys.\]
The non-convex constraint $X - \xs \xs^\top = 0$ is relaxed to $X - \xs \xs^\top \succeq 0$, which is equivalent to $\bigg(\begin{matrix}
X &  \xs \\
\xs^\top &  1
\end{matrix}\bigg) \succeq 0$ by the well-known Schur complement. Multiplying by $\rho > 0$ gives \[\bigg(\begin{matrix}
\Ys &  \ys \\
\ys^\top &  \rho
\end{matrix}\bigg) \succeq 0.\]
Furthermore, the objective function can be written as
\[
\frac{\xs^\top L \xs}{e^\top \xs} = \frac{\langle L, X \rangle}{e^\top \xs} = \langle L, \rho X \rangle = \langle L, \Ys \rangle.
\]
After putting these constraints together and adding $\bigg(\begin{matrix}
\Ys &  \ys \\
\ys^\top &  \rho
\end{matrix}\bigg) \ge 0$,
we arrive at the following basic DNN relaxation for the edge expansion problem.
\begin{equation}
\label{eq:sdp-relax-np1}
\tag{DNN\textsubscript{n+1}}
\begin{aligned}
\min        \quad & \langle \laplacian , \Ys \rangle \\
\text{s.t.} \quad & e^\top \ys = 1\\
& \frac{1}{\lfloor \frac n 2 \rfloor} \leq \rho \leq 1\\
&  1 \leq \langle E, \Ys \rangle \leq \Big\lfloor \frac n 2 \Big\rfloor\\
& \operatorname{diag}(\Ys) = \ys \\
& \bigg(\begin{matrix}
\Ys &  \ys \\
\ys^\top &  \rho
\end{matrix}\bigg)\in \text{DNN}.
\end{aligned}
\end{equation}
We call this a basic DNN relaxation since it
does not contain additional cutting planes such as triangle inequalities, etc.
Adding a rank-constraint to~\eqref{eq:sdp-relax-np1} results in computing $h(G)$.
We will later relate~\eqref{eq:sdp-relax-np1} to the relaxation which we
are going to derive in the next section. 

\section{Convexification of the edge expansion}\label{seq:theoretical-background}
In this section, we apply the recent convexification results of
He, Liu, and Tawarmalani in~\cite{he2024convexification} to reformulate the
edge expansion problem.
For this, we start by presenting the main results needed from that paper.
\subsection{A convexification technique for fractional programs}
\label{sec:theorems-fromhe2024}
The first important result on the convexification for fractional programs is the following.
\begin{theorem}[Theorem 2 of~\cite{he2024convexification}]
	\label{thm:thm2-he2023}
	Let $f \colon \mathcal X \subseteq \R^n \to \R^m$ with
	$f(x) = (f_1(x),\dots,f_m(x))^\top $ be a vector of base functions.
	And let another vector be obtained from the base functions by
	dividing each of them with a linear form of $f$.
	We then define the following two sets.
	\begin{align*}
		\mathcal F &= \{f(x) : x \in \mathcal X\} \\
		\mathcal G &= \bigg\{\frac{f(x)}{\sum_{i\in [m]} \alpha_i f_i(x)} : x \in \mathcal X \bigg\}
	\end{align*}
	We assume that $\mathcal F$ is bounded and there exists~$\varepsilon > 0$
	such that $\sum_{i\in [m]} \alpha_i f_i(x) > \varepsilon$ for all~$x \in \mathcal X$.
	Then,
	\begin{equation}
		\conv(\mathcal G) = \{ g \in \R^m : g \in \rho \conv(\mathcal F),\ \alpha^\top g = 1,\ \rho \geq 0 \}
	\end{equation}
	and if $f_1(x) = 1$, then
		\begin{equation}
	\conv(\mathcal F) = \{ f \in \R^m : f \in \sigma \conv(\mathcal G),\ f_1 = 1,\ \sigma \geq 0 \}.
	\end{equation}
\end{theorem}
This theorem is used to reformulate quadratic fractional binary
optimization problems to optimization problems over the completely
positive cone. 
Let $\mathcal{X} \subseteq \R^n$ and $\mathcal L$ be an affine set of $\R^n$.
We consider the following minimization problem
\begin{equation}
	\label{eq:general-quadr-frac-min}
	\min_{x \in \mathcal X \cap \mathcal L} \ 
	\frac{x^\top B x + b^\top x + b_0}{x^\top A x + a^\top x + a_0}
\end{equation}
and set $q(x) \coloneqq x^\top A x + a^\top x + a_0$ to be the denominator
of the objective function.
We assume that $\mathcal X \cap \mathcal L$ is non-empty, bounded and that
for all elements $x \in \mathcal X \cap \mathcal L$ it holds~$q(x) > 0$.
Let further
\begin{equation*}
	\mathcal G' =
	\bigg\{ \frac{(1,x,xx^\top)}{q(x)} : x \in \mathcal X \cap \mathcal L \bigg\},
\end{equation*}
then~\eqref{eq:general-quadr-frac-min} is equivalent to
\begin{equation}
	\min \big\{\langle B, Y \rangle + b^\top y + b_0 \rho
		    : (\rho, y, Y) \in \mathcal G' \big\}.
\end{equation}
One could equivalently also optimize the linear objective over the convex
hull~$\conv(\mathcal G')$.
To do so, the authors of~\cite{he2024convexification} showed an equivalence
between the convex hull of $\mathcal G'$ and the convex hull of the
set $\mathcal F$ with
\begin{equation*}
	\mathcal F = \{(1,x, X) : x \in \mathcal X, X = xx^\top \}.
\end{equation*}
Note, that in~\cite{he2024convexification} they omit the constant~$1$ and write
$\mathcal F = \{(x, xx^\top) : x \in \mathcal X\}$. For a better
understanding, we stick to the notations of \Cref{thm:thm2-he2023} with
our definition of~$\mathcal{F}$, including the constant~$1$.

Since the feasible set~$\mathcal X$ in~$\mathcal G'$ is, in contrast to~$\mathcal F$,
also intersected with the affine set~$\mathcal L$,
\Cref{thm:thm2-he2023} can not be directly applied but needs some facial decomposition.
\begin{proposition}[Proposition 3 in~\cite{he2024convexification}]
	\label{prop:conv-intersecwith-linear-eq}
	Assume that $q(x) > 0$ over a non-empty and bounded set
	$\mathcal X \cap \mathcal L \subseteq \R^n$ and suppose
	that $\mathcal L = \{x \in \R^n : Cx = d\}$ for some
	$C \in \R^{p \times n}$ and $d \in \R^p$.
	Then,
	\begin{align*}
		\conv(\mathcal G') = \Big\{ (\rho, y, Y) :\ 
				&(\rho, y, Y) \in \rho \conv(\mathcal F),\\
				&\langle A, Y\rangle + a^\top y + a_0\rho = 1,\ \rho \geq 0,\\
				&\tr(CYC^\top - Cyd^\top - dy^\top C^\top + \rho dd^\top) = 0 \Big\}.
	\end{align*}
	For $\mathcal L = \R^n$, additionally 
	$\conv(\mathcal F) = \big\{ (1,x,X) :
	(1,x,X) \in \sigma \conv(G'), \ \sigma \geq 0\big\}$ holds.
\end{proposition}

Depending on~$\mathcal X$, it still remains to characterize the
convex hull of~$\mathcal F$.
In the case of~$\mathcal X = \R_{\geq 0}^n$ for example, we have that
$ \mathcal F = \{(1, x, X) : x \in \R_{\geq 0}^n,\ X = xx^\top \}$
and hence the convex hull of $\mathcal F$ is
\begin{equation*}
	\conv(\mathcal F) = \bigg\{ (1, x, X) :
	\roundbr[\bigg]{\begin{matrix}
	X & x\\
	x^\top & 1
	\end{matrix} }
	\in \mathcal{CP}^{n+1}\bigg\},
\end{equation*}
where $\mathcal{CP}^{n+1}$ denotes the cone of completely
positive matrices of dimension $(n+1)\times(n+1)$.
This immediately allows us to state the following corollary.

\begin{corollary}[Corollary 3 in~\cite{he2024convexification}]
        If $\mathcal X$ is the non-negative orthant,
        and suppose
	that $\mathcal L = \{x \in \R^n : Cx = d\}$ for some
	$C \in \R^{p \times n}$ and $d \in \R^p$,
	then~\eqref{eq:general-quadr-frac-min} can be formulated as
	\begin{equation}
	\begin{aligned}
	\min \quad & \langle B, Y \rangle + b^\top y + b_0 \rho \\
	\text{s.t.} \quad & \langle A, Y \rangle + a^\top y + a_0 \rho = 1\\
	&\tr(CYC^\top - Cyd^\top - dy^\top C^\top + \rho dd^\top) = 0  \\
	& \roundbr[\bigg]{ \begin{matrix}
		Y & y\\
		y^\top & \rho
	  \end{matrix} }
	  \in \mathcal{CP}^{n+1}.
	\end{aligned}
	\end{equation} 
\end{corollary}

To handle binary variables in this setting, one can introduce
non-negative slack variables $\zs_i$ for each original variable $\xs_i$ and require $\xs_i + \zs_i = 1$.
We collect all these variables in a vector $x^\top = \roundbr{\begin{matrix}\xs^\top & \zs^\top \end{matrix}}$.
Let $k$ be the index of $\zs_i$ in the variable vector~$x$.
Since we are working with completely positive matrices, it can
be shown that by adding the constraint $X_{ik} = 0$, we get an exact reformulation,
as every extreme ray belonging to the face implied by this constraint
satisfies the binary condition, see Remark~4 in~\cite{he2024convexification}.

We are now ready to apply the theory of~\cite{he2024convexification} presented in this section
to the problem of computing the edge expansion of a graph.

\subsection{Applying the convexification techniques to the edge expansion}\label{sec:convexification}
In analogy to the general notation for the objective function of quadratic fractional problems
from \Cref{sec:theorems-fromhe2024} above, 
we get that for the edge expansion problem $B = \laplacian$, $b = 0$, $b_0 = 0$, and
$A = 0$, $a = e_n$, $a_0 = 0$.
In the next two subsections, we give two different ways to
reformulate~\eqref{eq:fracprobl-edgeexp}. The first one is
to apply~\Cref{thm:thm2-he2023} directly and the second one
is to formulate the problem in such a way that
\Cref{prop:conv-intersecwith-linear-eq} can be used. 
\subsubsection{Reformulation with \Cref{thm:thm2-he2023}}
Referring to the notation from before, we define
\begin{equation*}
\begin{aligned}
\mathcal G' &= \bigg\{ \frac{(1,\xs,\xs\xs^\top)}{e^\top\xs} :
		 \xs \in \{0,1\}^n,\ 1 \leq e^\top \xs \leq \Big\lfloor \frac n 2 \Big\rfloor \bigg\}, \text{ and}\\
		 \mathcal F' &=\bigg\{ (1,\xs,\xs\xs^\top) : 
		 \xs \in \{0,1\}^n,\ 1 \leq e^\top \xs \leq \Big\lfloor \frac n 2 \Big\rfloor \bigg\}.
\end{aligned}
\end{equation*}
With \Cref{thm:thm2-he2023} we get that
\begin{equation}
	\conv(\mathcal{G}') = \Big\{ (\rho, \ys, \Ys) :
						(\rho,\ys,\Ys) \in \rho \conv (\mathcal F'),\ 
						 e^\top \ys = 1,\ \rho \geq 0 \Big\}
\end{equation}
and $h(G) = \min \{ \langle L, \Ys \rangle :
					(\rho, \ys, \Ys) \in \conv(\mathcal G') \}$
holds.
Consequently, we are now interested in describing $\conv(\mathcal F')$.
Motivated by the proof of \Cref{prop:conv-intersecwith-linear-eq}
in~\cite{he2024convexification}, we introduce the slack variables
$s$ and $t$ to rewrite
the linear inequalities as the equations
\begin{equation*}
	\underbrace{
		\begin{pmatrix}
			e^\top_n & 1 & 0 \\
			e^\top_n & 0 & -1
		\end{pmatrix}
	}_{\Cm}
	\cdot
	\underbrace{
		\begin{pmatrix}
			\xs \\
			s \\
			t 
		\end{pmatrix}
	}_{\xm}
	=
	\underbrace{
		\begin{pmatrix}
			\big \lfloor \frac n 2 \big \rfloor \\
			1
		\end{pmatrix}
	}_{\dm}.
\end{equation*}
It then holds that
\begin{equation*}
	\conv(\mathcal F') = \Big\{ (1, \xm, \Xm) :
			\tr(\Cm\Xm \Cm^\top - \Cm \xm \dm^\top - \dm\xm^\top \Cm^\top + \dm\dm^\top) = 0,\ 
			(1,\xm, \Xm) \in \conv(\mathcal F)  \Big\}
\end{equation*}
with
\begin{equation*}
	\mathcal F = \Big\{ (1,\xm, \Xm) :  \Xm = \xm \xm^\top,\ 
					  \xm^\top = (\begin{matrix}\xs^\top & s & t \end{matrix}),
					  \ \xs \in \{0,1\}^n,\ s, t \in \R_{\geq 0}  
					    \Big\},
\end{equation*}
where the details of the proof can be found in~\cite{he2024convexification}.
Let \[\Lm = \operatorname{Diag}(\laplacian, 0_2),\]
then~\eqref{eq:fracprobl-edgeexp} is equivalent to
\begin{equation}
	\begin{aligned}
h(G) =	\min        \quad & \langle \Lm , \Ym \rangle \\
	    \text{s.t.} \quad & (\begin{matrix} e_n^\top & 0_2^\top \end{matrix}) \ym = 1\\
	& \tr(\Cm\Ym\Cm^\top - \Cm \ym \dm^\top - \dm \ym^\top \Cm^\top + \rho \dm\dm^\top) = 0\\
	& (\rho,\ym, \Ym) \in \rho \conv(\mathcal F)\\
	& \rho \geq 0.
	\end{aligned}
\end{equation}
The constraint
$\tr(\Cm\Ym \Cm^\top - \Cm \ym \dm^\top - \dm\ym^\top \Cm^\top + \rho\dm\dm^\top) = 0$
can also be written as the linear equality constraint
\[\inprod[\bigg]{\roundbr[\bigg]{\begin{matrix} \Cm^\top \\ - \dm^\top \end{matrix}}
			 (\begin{matrix} \Cm &  -\dm  \end{matrix}) ,
			 \roundbr[\bigg]{\begin{matrix}
			 	\Ym & \ym \\
			 	\ym^\top & \rho
			 \end{matrix}}
			   } = 0\]
in the matrix variable and is, therefore, easy to handle.
But still, the problem remains to optimize over the convex hull of $\mathcal F$.

In the next subsection, we use the more general Proposition~\ref{prop:conv-intersecwith-linear-eq}
to rewrite~\eqref{eq:fracprobl-edgeexp} as a completely positive optimization problem.

\subsubsection{Reformulation applying \Cref{prop:conv-intersecwith-linear-eq}}
Additionally to the slack variables $s, t \in \R_{\geq 0}$ for the two
inequalities of~\eqref{eq:fracprobl-edgeexp}, 
we now introduce for $\xs \in \R_{\geq 0}^n$ the slack variables $\zs \in \R_{\geq 0}^n$.
We denote by $\xb$ the vector collecting all variables, namely  
$\xb^\top = (\begin{matrix}\xs^\top & \zs^\top & s & t\end{matrix})$,
and set 
\begin{equation*}
	\Cb = \begin{pmatrix}
			e_n^\top & 0_n^\top & 1 & 0 \\
			e_n^\top & 0_n^\top & 0 & -1 \\
			I_n      & I_n & 0_n & 0_n
	\end{pmatrix}\in \R^{(n+2)\times (2n+2)} \hspace*{0.5cm} \text{and} \hspace*{0.5cm}\ 
	\db = \begin{pmatrix}
		\lfloor \frac n 2 \rfloor \\
		1  \\
		e_n
	\end{pmatrix} \in \R^{n+2}
\end{equation*}
to formulate the affine set as
$
	\mathcal L = \{ \xb \in \R^{2n + 2} : \Cb \xb = \db \}.
$
We further set the matrix variable~$\Yb$ and the vector variable~$\yb$
to be of the form
\begin{equation*}
\Yb =
\begin{pmatrix}
\Yb^{11} & \Yb^{12} & \Yb^{13} & \Yb^{14} \\
\Yb^{21} & \Yb^{22} & \Yb^{23} & \Yb^{24} \\
\Yb^{31} & \Yb^{32} & \Yb^{33} & \Yb^{34} \\
\Yb^{41} & \Yb^{42} & \Yb^{43} & \Yb^{44}
\end{pmatrix}
\in \symmmat{2n + 2}
\hspace*{0.5cm}
\text{and}
\hspace*{0.5cm}
\yb = \left(\begin{matrix} \yb^1 \\ \yb^2 \\ \yb^3 \\ \yb^4 \end{matrix}\right) \in \R^{2n+2},
\end{equation*}
that is $\Yb$ is a matrix given in block format with $\Yb^{11}, \Yb^{22} \in \symmmat{n}$
and $\Yb^{33}, \Yb^{44} \in \R$ and the vector~$\yb$ consists of $\yb^1, \yb^2 \in \R^n$
and $\yb^3, \yb^4 \in \R$.
For convenience of referring to specific parts of the matrix/vector, the structure
of~$\Yb \in \symmmat{2n + 2}$ and $\yb \in \R^{2n+2}$ is implicitly consistent throughout the rest
of this paper.

By Proposition~\ref{prop:conv-intersecwith-linear-eq} and
the remark after it holds that
\begin{equation}
	\label{eq:exact-form-2np3}
	\begin{aligned}
		h(G) =	\min        \quad & \langle \Lb , \Yb \rangle \\
		\text{s.t.} \quad & (\begin{matrix}
		e^\top_n & 0^\top_n & 0^\top_2
		\end{matrix}) \yb = 1\\
		&  \tr(\Cb \Yb \Cb^\top - \Cb \yb \db^\top - \db \yb^\top \Cb^\top + \rho \db\db^\top) = 0\\
		& \operatorname{diag}(\Yb^{12}) = 0 \\
		& \bigg(\begin{matrix}
		\Yb &  \yb \\
		\yb^\top & \rho
		\end{matrix}\bigg)\in \mathcal{CP}^{2n + 3},
	\end{aligned}
\end{equation}
where $\Lb = \operatorname{Diag}(\laplacian, 0_{n+2})$ or $\Lb = \frac 1 2
\operatorname{Diag}(I_2 \otimes \laplacian, 0_{2})$.
Note the similarity to the vector lifting procedure applied
in~\cite{wolkowicz-zhao1999} to the bisection problem. There, two
vectors, each indicating the vertices in the first and second
partition, respectively, are lifted into the space of $(2n + 1) \times
(2n + 1)$.

In the following sections of this paper, we are going
to consider the exact reformulation~\eqref{eq:exact-form-2np3}.

\section{A DNN relaxation of the reformulated problem}\label{sec:dnn}
To optimize over the cone of completely positive matrices is NP-hard.
But we can use the fact that every completely positive matrix is 
doubly non-negative, i.e., it is positive semidefinite and its entries are
non-negative, to derive a DNN relaxation of~\eqref{eq:exact-form-2np3}, which
is
\begin{equation}
\label{eq:sdp-relax-2np3}
\tag{DNN-P}
		\begin{aligned}
	\min        \quad & \langle \Lb , \Yb \rangle \\
	\text{s.t.} \quad & (\begin{matrix} e^\top_n & 0^\top_{n+2}
										\end{matrix}) \yb = 1\\
			&  \tr(\Cb \Yb \Cb^\top - \Cb \yb \db^\top - \db \yb^\top \Cb^\top + \rho \db\db^\top) = 0\\
			& \operatorname{diag}(\Yb^{12}) = 0 \\
			& \Yxl = \bigg(\begin{matrix}
			\Yb &  \yb \\
			\yb^\top & \rho
			\end{matrix}\bigg)\in \text{DNN}.
	\end{aligned}
\end{equation}

Let $M = (\begin{matrix}\Cb & - \db \end{matrix})$,
then we can rewrite the trace constraint in~\eqref{eq:sdp-relax-2np3} as
\begin{equation*}
	\tr(\Cb\Yb \Cb^\top - \Cb \yb \db^\top - 
	\db \yb^\top \Cb^\top + \rho \db \db^\top) = 
	\tr\big(
	M
	\Yxl
	M^\top
	\big) = 
	\tr\big(
		M^\top
		M
		\Yxl
		\big)
	= 0.
\end{equation*}
One can easily see that there exists no positive definite
matrix~$\Yxl$ fulfilling the trace constraint, as for any positive definite matrix~$\Yxl$ we have $\tr(M \Yxl M^\top) > 0$.
Hence,~\eqref{eq:sdp-relax-2np3} has no Slater point.
To obtain strict feasibility and reduce the dimension of the matrix variable~$\Yxl$, we perform a
facial reduction.

\subsection{Facial reduction}
In this section, we apply facial reduction to~\eqref{eq:sdp-relax-2np3} in order
to obtain a formulation with a Slater feasible point.
For this, we first state the following helpful proposition, which is a variant of Proposition~1
in~\cite{Burer2010} but with~$\rho$ in the bottom-right corner of~$\Yxl$ instead of~1.
\begin{proposition}
	\label{prop:MYMt}
	Let $M = ( \begin{matrix} \Cb & - \db \end{matrix} )$ and
	\begin{equation*}
		\Yxl = \bigg( \begin{matrix}
			\Yb & \yb \\
			\yb^\top & \rho
		\end{matrix} \bigg) \succeq 0,
	\end{equation*}
	then the following three statements are equivalent.
	\begin{enumerate}
		\item $\Cb\yb = \rho\db$ and $\diag(\Cb \Yb \Cb^\top) = \rho\db^2$, where the square 
			  is to be interpreted element-wise.
		\item $M \Yxl M^\top = 0$, or equivalently $\tr(M \Yxl M^\top) = 0$.
		\item $M \Yxl = 0$.
	\end{enumerate}
\end{proposition}
\begin{proof}
	Assume that $\Cb \yb = \rho\db$ and $\diag(\Cb\Yb\Cb^\top) = \rho\db^2$ holds.
	Let $c_i^\top$ denote the~$i$-th row of~$C$. It then holds that
	\begin{equation*}
		\roundbr{M\Yxl M^\top}_{ii} = \roundbr{\begin{matrix} c_i^\top & -\db_i \end{matrix}} \Yxl \roundbr[\bigg]{\begin{matrix} c_i \\ -d_i \end{matrix}} =
		c_i^\top \Yb c_i - 2 \db_i c_i^\top \yb + \rho \db_i^2 = \rho\db_i^2 - 2\rho\db_i^2 + \rho \db_i^2 = 0.
	\end{equation*}
	Hence, the diagonal and therefore also the trace of $M\Yxl M^\top$ is zero.
	Since~$\Yxl$ is positive semidefinite, also $M\Yxl M^\top$ is positive semidefinite, 
	leading to the conclusion that the equality $M\Yxl M^\top = 0$ has to hold.

	For the second part of the proof, assume  that $\tr(M\Yxl M^\top) = 0$.
	Since~$\Yxl \succeq 0$, we can write $\Yxl = UU^\top$ for some matrix $U$.
	Plugging in, we obtain
	\begin{equation*}
		0 = \tr(M \Yxl M^\top) = \tr(MUU^\top M^\top) = \norm{MU}^2,
	\end{equation*}
	which is equivalent to $MU = 0$ and therefor $M\Yxl = MUU^\top = 0$.

	Finally, assume that $M\Yxl = 0$ holds.
	To prove the equality $\Cb \yb = \rho \db$, we take a closer look at the last column of $M\Yxl$.
	The $i$-th entry in the last column is
	\begin{equation*}
		0 = \roundbr{\begin{matrix} c_i^\top & -\db_i \end{matrix}} \roundbr[\bigg]{\begin{matrix} \yb \\ \rho \end{matrix}} =
		    c_i^\top\yb - \rho \db_i,
	\end{equation*}
	yielding the desired equality $\Cb \yb = \rho \db$.
	From $M\Yxl = 0$ we get that $M\Yxl M^\top = 0$.
	To prove that $\diag(\Cb\Yb\Cb^\top) = \rho\db^2$ holds, we consider the diagonal of $M\Yxl M^\top$.
	For the $i$-th entry of the diagonal, it holds that
	\begin{equation*}
		0 = \roundbr{\begin{matrix} c_i^\top & -\db_i \end{matrix}} \Yxl \roundbr[\bigg]{\begin{matrix} c_i \\ -d_i \end{matrix}} = 
		    c_i^\top \Yb c_i - 2 \db_i c_i^\top \yb + \rho \db_i^2 = c_i^\top \Yb c_i - \rho \db_i^2.
	\end{equation*}
	Hence, $\diag(C\Yb C^\top)_i = c_i^\top \Yb c_i = \rho \db_i^2$ holds, which closes the proof.
\end{proof}
For the facial reduction of~\eqref{eq:sdp-relax-2np3}, let $W$ be a matrix
such that its columns form an orthonormal basis of the nullspace of~$M$.
It then holds that
\begin{equation}
	\{\Yxl \succeq 0 : M\Yxl = 0\} = \{WRW^\top : R \succeq 0\}.
\end{equation}
\begin{lemma}\label{lemma:basis}
The vectors
\begin{equation*}
w_{n+1} = \begin{pmatrix}
0_n\\
e_n\\
\lfloor \frac n 2 \rfloor\\
-1\\
1
\end{pmatrix} 
\text{ and }
w_i = \begin{pmatrix}
u_i\\
-u_i\\
-1\\
1\\
0
\end{pmatrix}
\text{ for } i \in \{1,\dots,n\}
\end{equation*}
form a basis of $\ker(M)$, where $u_i$ is the $i$-th unit vector with entry 1 at position~$i$ and 0 everywhere else.
\end{lemma}
\begin{proof}
  Since the rank of $M$ is $n + 2$, the dimension of $\ker(M)$ is $(2n + 3) - (n + 2) = n + 1$.
  Moreover, it is easy to check that the vectors $w_1,\dots,w_{n+1}$ are linearly independent
  and satisfy $Mw_i = 0$ for $i \in [n + 1]$.
\end{proof}
To obtain the required matrix $W\in \R^{(2n + 3) \times (n + 1)}$, one can take as columns of~$W$
the orthonormalized basis vectors
from \Cref{lemma:basis}.
We can then rewrite~\eqref{eq:sdp-relax-2np3} as
\begin{equation}
\label{eq:sdp-relax-facialred}
\tag{DNN-PFR}
\begin{aligned}
\min        \quad & \langle \Lxl , \Yxl \rangle \\
\text{s.t.} \quad & (\begin{matrix} e^\top_n & 0^\top_{n+2}
\end{matrix}) \yb = 1\\
& \operatorname{diag}(\Yb^{12}) = 0 \\
& \Yxl  =  
\bigg(\begin{matrix}
\Yb &  \yb \\
\yb^\top & \rho
\end{matrix}\bigg) = WRW^\top \geq 0\\
& R \succeq 0,
\end{aligned}
\end{equation}
with $\Lxl = \operatorname{Diag}(\Lb,0_{n+3})$.
We will show later in this section in~\Cref{prop:objective-bigsdp-same} that one can also equivalently
set~$\Lxl = \frac{1}{2}\operatorname{Diag}(I_2 \otimes L, 0_{3})$.
The reformulation~\eqref{eq:sdp-relax-facialred} is indeed a strictly feasible formulation of~\eqref{eq:sdp-relax-2np3}.
To prove this, we apply the following result of Hu, Sotirov, and Wolkowicz~\cite{Hu2023}
on the existence of a Slater feasible point for DNN relaxations.
\begin{theorem}[Theorem 3.15 in~\cite{Hu2023}]%
	\label{thm:slater-facialred-hu}
	Let
	\begin{equation*}
		\mathcal{Q} = \set[\Bigg]{\xb \in \R^{2n + 2} : 
							\mathcal{A}\roundbr[\bigg]{\roundbr[\bigg]{
								\begin{matrix}
									\xb \xb^\top & \xb \\
									\xb^\top & 1
								\end{matrix}
							}} = 0,\ x \geq 0},
	\end{equation*}
	where $\mathcal{A}$ is a linear transformation,
	and suppose $\aff\roundbr{\conv\roundbr{\mathcal{Q}}} = \mathcal{L}$
	with $\dim\roundbr{\mathcal{L}} = n$. Then there exist $\Cb$ with
	full row rank and $\db$ such that
	\begin{equation*}
		\mathcal{L} = \set[\big]{\xb \in \R^{2n+2} : \Cb\xb = \db}.
	\end{equation*}
	Let $M = \roundbr{\begin{matrix} \Cb & -\db \end{matrix}}$ and $W$
	be a matrix such that its columns form a basis of $\ker\roundbr{M}$.
	Let further~$\mathcal{J} = \set[\big]{(i,j) : \xb_{i}\xb_{j} = 0 \ 
			\forall \xb \in \mathcal{Q}	}$
	and let $\mathcal{J}^\mathrm{c}$ be its complement.
	Then there exists a Slater point $\hat R$ for the set
	\begin{equation*}
		\hat{\mathcal{Q}}_R = \set[\Big]{R \in \symmmat{n+1} : R \succeq 0,\ 
				\roundbr[\big]{WRW^\top}_{\mathcal{J}^\mathrm{c}} \geq 0,\ 
				\roundbr[\big]{WRW^\top}_{\mathcal{J}} = 0,\ 
				\mathcal{A}\roundbr[\big]{WRW^\top} = 0}
	\end{equation*}
	of feasible points.
\end{theorem}
With the help of this, we can now state the following result on Slater feasibility.
\begin{theorem}%
	\label{thm:slater-facialred-ours}
	Relaxation~\eqref{eq:sdp-relax-facialred} has a Slater feasible point.
\end{theorem}
\begin{proof}
Let $\xb^\top = (\begin{matrix}\xs^\top & \zs^\top & s & t\end{matrix})$, then
\begin{align*}
\mathcal{Q} &= \set[\Big]{\xb \in \R^{2n+2}_{\geq 0} : 
\xs_i \zs_i = 0\ \forall i \in [n],\ 
\xs + \zs = e_n,\ 
e_n^\top\xs + s = \floor[\Big]{\frac n 2},\ 
e_n^\top\xs - t = 1
} \\
	&=  \set[\Big]{\xb \in \Z^{2n+2}_{\geq 0} : 
			\xs + \zs = e_n,\ 
			e_n^\top\xs + s = \floor[\Big]{\frac n 2},\ 
			e_n^\top\xs - t = 1
			} \\
		&= \set[\Big]{
			\xb \in \Z^{2n+2}_{\geq 0} : \Cb \xb = \db
		}.
\end{align*}
With Ghouila-Houri’s characterization of totally unimodular matrices
one can show that~$\Cb$ is totally unimodular and hence
$\conv\roundbr{\mathcal{Q}} = \set[\Big]{
			\xb \in \R^{2n+2}_{\geq 0} : \Cb \xb = \db}$.
One can further show that
\[\aff\roundbr{\conv\roundbr{\mathcal{Q}}} = \set[\big]{
			\xb \in \R^{2n+2} : \Cb \xb = \db},\]
which follows from the fact that
$\aff\roundbr{\set{x \in \R^{2n+2} : Ax \leq b}} = 
\set{x \in \R^{2n+2} : A^=x = b^=}$.
From the constraints we can see that the index set $\mathcal{J}$ (see \Cref{thm:slater-facialred-hu})
is $\mathcal{J} = \set{(i,n+i) : 1 \leq i \leq n} \cup 
\set{(n+j,j) : 1 \leq j \leq n}$.
In $\mathcal{A}\roundbr{X}$ we have the constraints $\Cb X_{:,(2n+3)} = \db$
and $X_{i,(n+i)} = 0$ for all $1 \leq i \leq n$.
On the matrices of the form $WRW^\top$,
the first type of constraints is redundant, since $M\roundbr{WRW^\top}_{:,(2n + 3)} = 0$
holds for all matrices $R$ as the columns of $W$ span the kernel of $M$.

The second kind of constraints is $WRW^\top_\mathcal{J} = 0$.
Hence, by~\Cref{thm:slater-facialred-hu}, there exists a matrix $\hat R \in \symmmat{n+1}$ such that
$\hat R \succ 0$, $(W\hat RW^\top)_{\mathcal{J}^\mathrm{c}} > 0$ and 
$\roundbr[\big]{W\hat RW^\top}_{\mathcal{J}} = 0$.
Let \[\kappa = \sum_{i = 1}^n (W\hat RW^\top)_{i,(2n+3)} > 0,\] then $\frac 1 \kappa \hat R$
is a strictly feasible solution of~\eqref{eq:sdp-relax-facialred}.
\end{proof}

\Cref{prop:MYMt} was not only helpful for the facial reduction, but we can also use it
to identify several redundant constraints, as shown in the next section.
\subsection{Properties of~\eqref{eq:sdp-relax-2np3} and~\eqref{eq:sdp-relax-facialred}}
In this section, we state several properties of~\eqref{eq:sdp-relax-2np3} and~\eqref{eq:sdp-relax-facialred}.
Some of these properties are useful for the algorithm that we derive in \Cref{sec:auglag}.

As already mentioned, we make in particular use of the
result from~\Cref{prop:MYMt} stating that every feasible solution of~\eqref{eq:sdp-relax-2np3}
satisfies $M \Yxl = 0$ and	$M \Yxl M^\top  = 0$.
Note, that for feasible solutions~$WRW^\top$ of~\eqref{eq:sdp-relax-facialred} it holds that~$MWRW^\top = 0$ and hence all
results in this section apply to both~$\Yxl$ and $WRW^\top$.
The first result focuses on the last row/column of feasible matrices,
in particular bounds on~$\rho$ and connections between the entries in~$\yb$
are given.
\begin{proposition}
	\label{prop:bounds-rho}
	Every feasible point of~\eqref{eq:sdp-relax-2np3} and~\eqref{eq:sdp-relax-facialred} satisfies
	\begin{equation*}
	\frac{1}{\lfloor \frac n 2 \rfloor} \leq \rho \leq 1.
	\end{equation*}
	Moreover, it holds that $\yb^1 \leq \rho$, $\yb^2 \leq \rho$, $\yb^2 = \rho e_n - \yb^1$,
	$\yb^3 = \rho \lfloor \frac n 2 \rfloor - 1$, and
	$\yb^4 = 1 - \rho$.
\end{proposition}
\begin{proof}
	From the last column of the matrix equality $M\Yxl = 0$ it follows that
	$\Cb \yb = \rho \db$ and hence
	\begin{align*}
	(\begin{matrix} e^\top_n & 0^\top_n & 1 & \phantom{-}0 \end{matrix}) \yb &=
	\rho \Big\lfloor \frac n 2 \Big\rfloor\\
	(\begin{matrix} e^\top_n & 0^\top_n & 0 & -1 \end{matrix}) \yb &=
	\rho \\
	(\begin{matrix} I_n & I_n & 0_n & 0_n \end{matrix}) \yb &= \rho e_n
	\end{align*}
	holds. Remember, that for simplicity we denoted $\yb = (\begin{matrix}\yb^1 & \yb^2 & \yb^3 & \yb^4\end{matrix})$
	with $\yb^1, \yb^2 \in \R^n$ and $\yb^3, \yb^4 \in \R$.
	From the constraint $e_n^\top \yb^1 = 1$ and the equations above, it follows 
	that~$\yb^3 = \rho \lfloor \frac n 2 \rfloor - 1$. By the non-negativity of $\yb^3$,
	we get that $\rho \geq \frac{1}{\lfloor \frac n 2 \rfloor}$ has to hold.
	With the same argumentation, we get that $\yb^4 = 1 - \rho \geq 0$ holds for
	every feasible point, proving the claimed upper bound on~$\rho$.
	From the last of the above derived constraints and non-negativity, we get
	$\yb^1 = \rho e_n - \yb^2 \geq 0$ and $\yb^2 = \rho e_n - \yb^1 \geq 0$,
	which implies that~$\yb^1 \leq \rho$ and $\yb^2 \leq \rho$ has to be satisfied.
\end{proof}

Similarly, we can show the following connections between the submatrices~$\Yb^{11}$, $\Yb^{22}$ and~$\Yb^{12}$ and the
vectors~$\yb^1$ and~$\yb^2$.
\begin{proposition}\label{prop:Y11plusY21}
	Every feasible solution~$\Yxl$ of~\eqref{eq:sdp-relax-2np3} and~\eqref{eq:sdp-relax-facialred} satisfies
	\begin{align*}
		\Yb^{11}_{ij} + (\Yb^{12})^\top_{ij} &= \yb^1_j \text{ and }\\
		\Yb^{22}_{ij} + \Yb^{12}_{ij} &= \yb^2_j
	\end{align*}
	for all $1 \leq i,j \leq n$.
\end{proposition}
\begin{proof}
	By considering the equations
	\begin{align*}
	0 &= (M\Yxl)_{(2n + i),j}
	= \Yb^{11}_{ij} + \Yb^{21}_{ij} - \yb^1_j
	\qquad\  \text{and}\\
	0 &= (M\Yxl)_{\roundbr{2n + i},(n + j)} =
	\Yb^{12}_{ij} + \Yb^{22}_{ij} - \yb^2_j
	\end{align*}
	for $1 \leq  i,j \leq n$, one immediately obtains the above stated result.	
\end{proof}

We can further show that the bounds on the sum of the entries in~$\Yb^{11}$ are already implied
by the constraints in~\eqref{eq:sdp-relax-2np3} and~\eqref{eq:sdp-relax-facialred} as well.
\begin{proposition}
	\label{prop:sumY11-inequality-implied}
	Every feasible point of~\eqref{eq:sdp-relax-2np3} and~\eqref{eq:sdp-relax-facialred}
	satisfies
	\begin{equation*}
	1 \leq
	\langle E, \Yb^{11} \rangle \leq \Big\lfloor \frac n 2 \Big\rfloor.
	\end{equation*}
\end{proposition}
\begin{proof}
	From the positive semidefiniteness of the matrix variable, it follows that
	\begin{align*}
	0 &\leq (\begin{matrix} e_n^\top & 0_{n+2}^\top & -1 \end{matrix})
	\bigg(\begin{matrix}
	\Yb &  \yb \\
	\yb^\top & \rho
	\end{matrix}\bigg)
	\Bigg(\begin{matrix}
	e_n \\ 0_{n+2} \\ -1
	\end{matrix}\Bigg)\\
	&= e_n^\top \Yb^{11} e_n 
	- 2 (\begin{matrix} e_n^\top & 0_{n+2} \end{matrix})^\top \yb + \rho\\
	&= \langle E, \Yb^{11} \rangle - 2 + \rho.
	\end{align*}
	Since $\rho \leq 1$, we obtain the lower bound $1 \leq \langle E, \Yb^{11}\rangle$.
	
	Considering \Cref{prop:bounds-rho}
	and
	\begin{equation*}
	0 = (M\Yxl)_{1,\roundbr{2n+1}} = 
	e_n^\top \Yb^{13} + \Yb^{33} -\Big\lfloor \frac n 2 \Big\rfloor \yb^3 = 
	e_n^\top \Yb^{13} + \Yb^{33} -  \Big\lfloor \frac n 2 \Big\rfloor^2\rho + 
	\Big\lfloor \frac n 2 \Big\rfloor,
	\end{equation*}
	yields
	\begin{equation*}
	e_n^\top \Yb^{13} + \Yb^{33} =  \Big\lfloor \frac n 2 \Big\rfloor^2\rho -
	\Big\lfloor \frac n 2 \Big\rfloor.
	\end{equation*}
	Plugging this into $(M\Yxl M^\top)_{1,1} = 0$, we can then derive
	\begin{align*}
	0 &= 
	e_n^\top \Yb^{11} e_n + \Yb^{33} + \Big\lfloor \frac{n}{2} \Big\rfloor^2 \rho
	+ 2 e_n^\top \Yb^{13} - 2 \Big\lfloor \frac{n}{2} \Big\rfloor e_n^\top \yb^1
	- 2  \Big\lfloor \frac{n}{2} \Big\rfloor \yb^3 \\
	&= 	e_n^\top \Yb^{11} e_n + \Yb^{33} + \Big\lfloor \frac{n}{2} \Big\rfloor^2 \rho
	+ 2 e_n^\top \Yb^{13} - 2 \Big\lfloor \frac{n}{2} \Big\rfloor 
	- 2  \Big\lfloor \frac{n}{2}
	\Big\rfloor\Big(\Big\lfloor \frac{n}{2} \Big\rfloor\rho - 1\Big) \\
	&= e_n^\top \Yb^{11} e_n + \Yb^{33} + 2 e_n^\top \Yb^{13} 
	- \Big\lfloor \frac{n}{2} \Big\rfloor^2 \rho\\
	&= e_n^\top \Yb^{11} e_n + e_n^\top \Yb^{13} -
	\Big\lfloor \frac n 2 \Big\rfloor.
	\end{align*}
	Due to the non-negativity, it holds that $e_n^\top \Yb^{13} \geq 0$  
	and therefore $\langle E, \Yb^{11} \rangle \leq \lfloor \frac n 2 \rfloor$.	
\end{proof}

In addition to the non-negativity, we can derive the following upper bounds on the entries of
all feasible matrices.
\begin{proposition}\label{prop:ub-entries-Y}
	Every feasible solution of~\eqref{eq:sdp-relax-2np3} and~\eqref{eq:sdp-relax-facialred} satisfies
	\begin{align*}
	\yb^3 &\leq \Big\lfloor\frac{n}{2}\Big\rfloor - 1,\\
	\yb^4 &\leq 1 - \frac{1}{\lfloor\frac{n}{2}\rfloor},\\
	\Yb^{33} &\leq \Big\lfloor \frac n 2 \Big\rfloor ^2 -
	\Big\lfloor \frac n 2 \Big\rfloor,\\
	\Yb^{44} &\leq \Big\lfloor \frac n 2 \Big\rfloor - 1,\\
	\Yb^{34} &\leq \Big\lfloor\frac{n}{2}\Big\rfloor - 1, \\
	\Yb^{13}, \Yb^{23} &\leq
	\Big\lfloor\frac{n}{2}\Big\rfloor -1 \text{, and }\\
	\Yb^{14}, \hat Y^{24} &\leq 1 - \frac{1}{\lfloor\frac{n}{2}\rfloor}.
	\end{align*}
\end{proposition}
\begin{proof}
	The first two inequalities follow directly from~\Cref{prop:bounds-rho}, namely
	\begin{align*}
	\yb^3 &= \rho \Big\lfloor\frac{n}{2}\Big\rfloor - 1 \leq \Big\lfloor\frac{n}{2}\Big\rfloor - 1
	\text{ and }\\
	\yb^4 &= 1 - \rho \leq 1 - \frac{1}{\lfloor\frac{n}{2}\rfloor}.
	\end{align*}
	From $0 = (M\Yxl)_{1,(2n + 1)} = e_n^\top \Yb^{13} + \Yb^{33} -
	\lfloor\frac{n}{2}\rfloor \yb^3$, we obtain
	\[\Yb^{33} = \Big\lfloor\frac{n}{2}\Big\rfloor \yb^3 - e_n^\top \Yb^{13} \leq
	\Big\lfloor \frac n 2 \Big\rfloor ^2 - 
	\Big\lfloor \frac n 2 \Big\rfloor,
	\]
	using the upper bound on $\yb^3$ and $\Yb \geq 0$.
	Similarly, we can derive the upper bound for~$\Yb^{34}$.
	It holds that
	$
	0 = (M\Yxl)_{1,(2n+2)} = e_n^\top \Yb^{14} + \Yb^{34} -
	\lfloor\frac{n}{2}\rfloor \yb^4,
	$
	hence \[\Yb^{34}  = \Big\lfloor\frac{n}{2}\Big\rfloor \yb^4 -
	e_n^\top \Yb^{14} \leq   \Big\lfloor\frac{n}{2}\Big\rfloor  - 1.\]

	For the upper bound on~$\Yb^{44}$, consider the equations
	\begin{align*}
	0 &= (M\Yxl)_{2,(2n+2)} = e_n^\top \Yb^{14} - \Yb^{44} - \yb^4 \Leftrightarrow e_n^\top \Yb^{14} = \Yb^{44} + \yb^4\\
	0 &= (M\Yxl M^\top)_{2,2} = \langle E, \Yb^{11} \rangle +
	\Yb^{44} + \rho - 2e_n^\top \Yb^{14} - 2 + 2\yb^4.
	\end{align*}
	Plugging in the first into the second equation, we get
	\begin{equation*}
	\Yb^{44} = \langle E, \Yb^{11} \rangle + \rho - 2 \leq
	\Big \lfloor\frac{n}{2}\Big\rfloor -1,
	\end{equation*}
	since $\langle E, \Yb^{11} \rangle \leq \lfloor\frac n 2 \rfloor$~(cf.~\Cref{prop:sumY11-inequality-implied}) 
	and $\rho \leq 1$.
	Next, observe that
	\begin{align*}
	0_n &= (M\Yxl)_{3:(n+2),(2n+1)} = \Yb^{13} + \Yb^{23} - \yb^3 e_n \text{ and}\\
	0_n &= (M\Yxl)_{3:(n+2),(2n+2)} = \Yb^{14} + \Yb^{24} - \yb^4 e_n
	\end{align*}
	imply that $(\Yb^{13} + \Yb^{23})_i = \yb^3$ and $(\Yb^{14} + \Yb^{24})_i = \yb^4$
	has to hold for all~$i \in [n]$.
	Due to the non-negativity of~$\Yb$, each of the summands on the left-hand side is bounded by the
	right-hand side of the equations.
\end{proof}

We have already shown that~$\Yb^{11}$ and $\rho$ satisfy the inequality constraints of
the basic relaxation~\eqref{eq:sdp-relax-np1}.
In the subsequent section, we now want to strengthen our new relaxation~\eqref{eq:sdp-relax-2np3}
with further valid (in)equalities and compare it to the basic relaxation.
\subsection{Strengthening and comparison to the basic DNN relaxation}\label{sec:simple}
Recall, that the constraint $\Yxl \in \text{DNN}$ is the relaxation of
$\Yxl \in \mathcal{CP}^{2n+3}$, which is equivalent to 
$(\rho,\yb, \Yb) \in \rho \conv(\{ (1,\xb, \xb\xb^\top) : \xb \in \R^{2n + 2}_{\geq 0} \})$,
where we set $\xb^\top = (\begin{matrix}\xs^\top & \zs^\top & s & t\end{matrix})$.
Since the variables $\xs$ and $\zs$ are supposed to be binary,
we can add the following diagonal constraints.
\begin{equation}
	\label{eq:diagonalcon}
	\diag(\Yb^{11}) = \yb^1,\ \diag(\Yb^{22}) = \yb^2.
\end{equation}
These constraints can be added to~\eqref{eq:sdp-relax-facialred} without losing strict feasibility,
as presented in the following proposition.
\begin{proposition}
	Relaxation~\eqref{eq:sdp-relax-facialred} with the additional diagonal constraints~\eqref{eq:diagonalcon}
	has a Slater feasible point.
\end{proposition}
\begin{proof}
Note, that in the proof of~\Cref{thm:slater-facialred-ours} we can add to the description
of~$\mathcal{Q}$ the constraint $\xs_i \xs_i = \xs_i$ and
$\zs_i \zs_i = \zs_i$ for all $1 \leq i \leq n$ without any impact on $\aff\roundbr{\conv(\mathcal{Q})}$.
Hence,~\eqref{eq:sdp-relax-facialred} with the diagonal constraint~\eqref{eq:diagonalcon}
has a Slater feasible point as well.
\end{proof}
To be able to compare~\eqref{eq:sdp-relax-2np3} to the basic relaxation~\eqref{eq:sdp-relax-np1},
we first present the following result on the objective function.
\begin{proposition}%
	\label{prop:objective-bigsdp-same}
	Every feasible point~$\Yxl$ of~\eqref{eq:sdp-relax-2np3} and~\eqref{eq:sdp-relax-facialred}
	satisfies
	\begin{equation*}
	\Yb^{22} = \Yb^{11} + \rho E - e_n \roundbr{\yb^1}^\top - \yb^1 e_n^\top
	\end{equation*}
	and therefore
	\begin{equation*}
	\big\langle \Yb, \operatorname{Diag}(\laplacian, 0_{n+2}) \big\rangle = 
	\frac{1}{2}\big\langle \Yb, \operatorname{Diag}(I_2 \otimes L, 0_{2}) \big\rangle.
	\end{equation*}
\end{proposition}
\begin{proof}
	Using \Cref{prop:Y11plusY21} and then \Cref{prop:bounds-rho}, we get 
	\begin{align*}
		\Yb^{11}_{ij} + \Yb^{11}_{ji} - \Yb^{22}_{ij} - \Yb^{22}_{ji} &= \yb^1_j - \Yb^{12}_{ji} + \yb^1_i - \Yb^{12}_{ij} - \yb^2_j + \Yb^{12}_{ij} - \yb^2_i + \Yb^{12}_{ji}\\
		&= \yb^1_j + \yb^1_i - (\rho - \yb^1_j) - (\rho - \yb^1_i).
	\end{align*}
	From the symmetry of 
	$\Yxl$ it follows that $\Yb^{11}, \Yb^{22} \in \symmmat{n}$ and hence,
	\begin{equation*}
		\Yb^{22}_{ij} = \Yb^{11}_{ij} + \rho - \yb^1_i
		- \yb^1_j.
		\end{equation*}
	It then holds that $\inprod{L,\Yb^{22}} = \inprod{L,\Yb^{11}}$ due to the properties of the
	Laplacian matrix~$L$, which closes our proof.
\end{proof}

\begin{theorem}
	\label{thm:sdp2domsdp1}
	Relaxation~\eqref{eq:sdp-relax-2np3} with the diagonal constraint
	\begin{equation}\label{eq:diagconY11only}
		\diag(\Yb^{11}) = \yb^1
	\end{equation}
	is at least as good as~\eqref{eq:sdp-relax-np1}.
\end{theorem}
\begin{proof}
	To prove the theorem, we construct for every feasible point of the formulation~\eqref{eq:sdp-relax-2np3}~+~\eqref{eq:diagconY11only}
	a feasible point for~\eqref{eq:sdp-relax-np1} with the same objective function value.
	Let~$\Yxl \in \symmmat{2n + 3}$ be a matrix satisfying
	all constraints in~\eqref{eq:sdp-relax-2np3} and~\eqref{eq:diagconY11only}.
	From~\Cref{prop:bounds-rho,prop:sumY11-inequality-implied} and the
	constraint~\eqref{eq:diagconY11only}, we get that the submatrix
	\begin{equation*}
		\bigg(\begin{matrix}
		\Yb^{11} &  \yb^1 \\
		(\yb^1)^\top & \rho
		\end{matrix}\bigg)
	\end{equation*}
	satisfies all constraints of~\eqref{eq:sdp-relax-np1}.
	Due to~\Cref{prop:objective-bigsdp-same}, it holds that the
	objective function values are the same and hence~\eqref{eq:sdp-relax-2np3} is at least as good as~\eqref{eq:sdp-relax-np1}.
\end{proof}
\begin{remark}
\Cref{thm:sdp2domsdp1} also works without adding the constraint~$\diag(\Yb^{11}) = \yb^1$ if we remove the constraint
$\diag(\Ys) = \ys$ from~\eqref{eq:sdp-relax-np1}, i.e.,~\eqref{eq:sdp-relax-2np3}
is at least as good as~\eqref{eq:sdp-relax-np1} without the diagonal constraint.
\end{remark}
With the same argumentation as for the diagonal constraints,
we can further strengthen relaxation~\eqref{eq:sdp-relax-2np3}
by adding scaled facet defining inequalities for the boolean
quadric polytope (BQP) on
$\yb^1$, $\yb^2$ and
the left upper $2n \times 2n$ block of~$\Yxl$.
The scaled (multiplied with $\rho$) BQP inequalities are
\begin{subequations}
	\begin{align}
	\Yb_{ij} &\leq \yb_i \label{subeq:bqpa}\\
	\Yb_{ij} + \Yb_{ik} - \Yb_{jk} & \leq \yb_i \label{subeq:bqpb} \\
	\yb_i + \yb_j - \Yb_{ij} &\leq \rho \label{subeq:bqpc} \\
	\yb_i + \yb_j + \yb_k - \Yb_{ij} - \Yb_{ik} - \Yb_{jk} &\leq \rho \label{subeq:bqpd}
	\end{align}
\end{subequations}
for all $1 \leq i, j, k \leq 2n$.
Indeed, similarly as in~\cite{Sotirov2017} for the vector-lifted DNN relaxation of the graph bisection problem,
 we can show that constraints~\eqref{subeq:bqpa} and~\eqref{subeq:bqpc} are already implied
 by the constraints in~\eqref{eq:sdp-relax-2np3} and~\eqref{eq:sdp-relax-facialred}.
\begin{proposition}\label{prop:impliedbqpineq}
	Every feasible solution~$\Yxl$ of~\eqref{eq:sdp-relax-2np3} and~\eqref{eq:sdp-relax-facialred} satisfies
	the scaled BQP inequalities~\eqref{subeq:bqpa} and~\eqref{subeq:bqpc}.
\end{proposition}
\begin{proof}
	From~\Cref{prop:Y11plusY21} and $\Yxl \geq 0$ it follows that $\Yb_{ij} \leq \yb_j$ 
	for all $1 \leq i,j \leq 2n$, which is~\eqref{subeq:bqpa}.

	By~\Cref{prop:objective-bigsdp-same} and $\Yb^{22} \geq 0$ we get that
	$\Yb^{11}_{ij} + \rho - \yb^1_i - \yb^1_j \geq 0$ holds, which is equivalent to~\eqref{subeq:bqpc}
	for~$1 \leq i,j \leq n$.
	For~$1 \leq i \leq n$ and $(n+1)\leq j \leq 2n$, let $k = j - n$, then~\eqref{subeq:bqpc}
	can be written as
	\begin{align*}
		\rho &\geq \yb^1_i + \yb^2_k - \Yb^{12}_k = \yb^1_i + \Yb^{22}_{ik}
		= \Yb^{11}_{ik} - \yb^1_k + \rho,
	\end{align*}
	where the first equality follows from~\Cref{prop:Y11plusY21} and the second follows from~\Cref{prop:objective-bigsdp-same}.
	Hence, for this choice of~$i$ and~$j$,~\eqref{subeq:bqpc} is equivalent to~\eqref{subeq:bqpa}. In the same way, one can prove the case~$(n + 1) \leq i \leq 2n$ and $1 \leq j \leq n$.
	For the case~$(n + 1) \leq i,j \leq 2n$ let $k = n - i$ and $\ell = n - j$.
	Applying~\Cref{prop:bounds-rho} and~\Cref{prop:objective-bigsdp-same} we obtain
	that~\eqref{subeq:bqpc} can be written as
	\begin{equation*}
		\rho \geq \yb^2_k + \yb^2_\ell - \Yb^{22}_{k\ell} = 2 \rho - \yb^1_k - \yb^1_\ell - \Yb^{22}_{k\ell}
		= \rho - \Yb^{11}_{k\ell},
	\end{equation*}
	which is equivalent to~$\Yb^{11} \geq 0$ and thus valid for every feasible solution of~\eqref{eq:sdp-relax-2np3}
	and~\eqref{eq:sdp-relax-facialred}.
\end{proof}

\begin{corollary}
	\label{thm:sdp2domsdp1plusbqpac}
	Relaxation~\eqref{eq:sdp-relax-2np3} with the diagonal constraint
	\begin{equation*}
		\diag(\Yb^{11}) = \yb^1
	\end{equation*}
	is at least as good as~\eqref{eq:sdp-relax-np1} with the additional constraints
	\begin{align*}
		\Ys_{ij} &\leq \ys_i \\
		\ys_i + \ys_j - \Ys_{ij} &\leq \rho
	\end{align*}
	for $1 \leq i,j \leq n$.
\end{corollary}
\begin{proof}
	The statement follows from the proof of \Cref{thm:sdp2domsdp1} and \Cref{prop:impliedbqpineq}.
\end{proof}

Our numerical tests suggest that~\eqref{subeq:bqpb} leads to a better
improvement of the DNN relaxation than~\eqref{subeq:bqpd}.
The diagonal constraint yields a small improvement only, and therefore we
will not consider it for our relaxation.

Solving the relaxation with non-negativity and scaled triangle inequalities with
interior point solvers in a reasonable time is not possible due to the number of constraints.
In the next section, we therefore introduce an augmented Lagrangian algorithm to compute
the strengthened DNN relaxation.

\section{Solving the DNN relaxation}\label{sec:auglag}
\subsection{An augmented Lagrangian algorithm}\label{sec:auglag-subsec}
Problem~\eqref{eq:sdp-relax-facialred} strengthened by a subset~$\mathcal{T}$ of the scaled BQP inequalities~\eqref{subeq:bqpb}
can be written as 
\begin{equation}
\label{eq:DNN-CUTS-P-facialred}
\tag{DNN-PFRC}
\begin{aligned}
 \min      \quad & \inprod[\big]{W^\top \Lxl W , R} \\
\text{s.t.} \quad & \mathcal A\roundbr[\big]{WRW^\top} = b \\
& \mathcal{B}\roundbr[\big]{WRW^\top} \leq 0 \\
&WRW^\top \ge 0\\
& R \succeq 0,
\end{aligned}
\end{equation}
where $\mathcal A \colon \symmmat{2n + 3} \to \mathbb{R}^{p}$ is the operator corresponding to the
linear equality constraints from~\eqref{eq:sdp-relax-facialred} and
$\mathcal{B} \colon \symmmat{2n + 3} \to \mathbb{R}^{q}$ is the operator corresponding to the BQP
inequalities in~$\mathcal T$.
Note that~$p$ denotes the number of equality constraints and is equal to~$3n + 1$ if the diagonal constraint is
included and~$n+1$ otherwise. The number of BQP inequalities in~$\mathcal T$ is denoted by~$q$.
We define the Lagrangian function with respect to the primal
variable $R\succeq 0$ and the dual variables $\dualvarone \in \mathbb{R}^{p}$, $\dualvartwo \in \mathbb{R}^{q}$, 
$S \in \symmmat{2n + 3}$ and $Z \in \symmmat{n + 1}$ corresponding to the equality constraints, inequality constraints,
non-negativity and positive semidefiniteness constraint,  as
\begin{align*}
\mathcal{L}(R; \dualvarone, \dualvartwo, S, Z) &= \inprod[\big]{W^\top \Lxl W , R} +
 \dualvarone^\top\roundbr[\big]{b - \mathcal A\roundbr[\big]{WRW^\top}} + \dualvartwo^\top \mathcal{B}\roundbr[\big]{WRW^\top} \\
  & \qquad - \inprod[\big]{WRW^\top, S} - \inprod{R,  Z}\\
&= b^\top \dualvarone - \inprod[\big]{W^\top \roundbr[\big]{\mathcal A^\top \dualvarone - \mathcal B^\top \dualvartwo + S - \Lxl} W + Z, R}.
\end{align*}
The Lagrange dual function is then defined by
\begin{equation*}
	g(\dualvarone, \dualvartwo, S, Z) = \inf_{R \in \symmmat{n + 1}} \mathcal{L}(R; \dualvarone,  \dualvartwo, S,  Z) = \begin{cases}
		b^\top \dualvarone & \text{if } W^\top\roundbr[\big]{\mathcal{A}^\top \dualvarone - \mathcal B^\top \dualvartwo + S - \Lxl}W + Z = 0,\\
		-\infty & \text{else,}
	\end{cases}
\end{equation*}
and the dual problem of~\eqref{eq:DNN-CUTS-P-facialred} is given by
$\max_{\dualvarone,\dualvartwo \geq 0, S\geq 0, Z\succeq 0} g(\dualvarone, \dualvartwo, S, Z)$, that is
\begin{equation}
\tag{DNN-DFRC}
\label{eq:dual_sdp-relax-facialred}
\begin{aligned}
 \max     \quad & b^\top \dualvarone \\
\text{s.t.} \quad & W^\top\roundbr[\big]{\mathcal{A}^\top \dualvarone - \mathcal B^\top \dualvartwo + S - \Lxl}W + Z = 0\\
&\dualvarone\in \R^p,\ \dualvartwo \geq 0 ,\ S \ge 0,  \ Z \succeq 0. \\
\end{aligned}
\end{equation}

We propose to use the augmented Lagrangian approach to approximately solve the above dual problem. 
By introducing a Lagrange multiplier $R$ for the dual equality constraint,  and a penalty parameter $\alpha > 0$, 
the augmented Lagrangian function $\mathcal{L}_{\alpha}$ can be written as
{
\begin{align*}
\mathcal{L}_{\alpha}(\dualvarone,\dualvartwo,S,Z{;} R) &= b^\top \dualvarone 
	-\inprod[\big]{W^\top\roundbr[\big]{\mathcal A^\top \dualvarone - \mathcal B^\top \dualvartwo + S - \Lxl} W + Z, R} \\
&\qquad - \frac{1}{2\alpha}\norm[\Big]{W^\top\roundbr[\big]{\mathcal A^\top \dualvarone - \mathcal B^\top\dualvartwo + S -  \Lxl}W + Z}^2\\
&=b^\top \dualvarone  - \frac{1}{2\alpha}\norm[\Big]{W^\top\roundbr[\big]{\mathcal A^\top \dualvarone -\mathcal B^\top\dualvartwo+ S -  \Lxl} W + Z + \alpha R}^2 + 
    \frac{\alpha}{2}\norm{R}^2.
\end{align*}}%
Note that the Lagrangian dual of~\eqref{eq:dual_sdp-relax-facialred} is again~\eqref{eq:DNN-CUTS-P-facialred},
which justifies our choice of the dual variable name~$R$.
The augmented Lagrangian method for solving~\eqref{eq:dual_sdp-relax-facialred} consists in maximizing
$\mathcal{L}_{\alpha}(\dualvarone, \dualvartwo,S,Z; R_k)$ for a fixed $\alpha > 0$ and $R_k$ to get $\dualvarone_k \in \mathbb{R}^{p}$,
$\dualvartwo_k \ge 0$, $S_k \ge 0$ and $Z_k \succeq 0$. 
Then the primal matrix $R$ is updated using 
\begin{equation}
\label{eq:update_primal_matrix}
R_{k+1} = R_k - \frac{1}{\alpha}\roundbr[\big]{W^\top \roundbr[\big]{\mathcal A^\top \dualvarone_k - \mathcal B^\top\dualvartwo_k + S_k - \Lxl }W + Z_k},
\end{equation}
see~\cite{bertsekas2014constrained}.  By construction, the primal matrix $\Yxl$ is then given by
\begin{equation}
\label{eq:construct_primal_matrix}
\Yxl = WRW^\top.
\end{equation}
In the augmented Lagrangian method, 
as opposed to the penalty method,  the penalty parameter $\alpha$ does not necessarily need to go to zero in order
to guarantee convergence.
However, to avoid problem-specific tuning of the penalty parameter, we let $\alpha \to 0$.
Decreasing $\alpha$ makes the subproblem in each iteration harder to solve due to ill-conditioning.
First-order methods are particularly sensitive to ill-conditioning, hence, we use here a quasi-Newton method to solve the inner problem.   

For given $\alpha > 0$ and $R \succeq 0$, the inner maximization problem
\begin{align}
\label{eq:inner_problem}
\begin{split}
\max \quad&  \ b^\top \dualvarone  - \frac{1}{2\alpha}\norm[\Big]{W^\top\roundbr[\big]{\mathcal A^\top \dualvarone - \mathcal B^\top\dualvartwo + S -  \Lxl} W + Z + 
					\alpha R}^2 + \frac{\alpha}{2} \norm{R}^2 \\
\textrm{s.t.} \quad&  \dualvarone \in \mathbb{R}^{p}, \dualvartwo \geq 0, \ S \ge 0, \ Z \succeq 0
\end{split}
\end{align}
can be further simplified by eliminating the matrix $Z$ as follows.  Define
\[
M = W^\top\roundbr[\big]{\mathcal A^\top \dualvarone - \mathcal B^\top\dualvartwo + S - \Lxl}W + \alpha R.
\]
For fixed $\dualvarone$, $\dualvartwo$ and $S$, the optimal~$Z$ of problem~\eqref{eq:inner_problem} is the
same as
\begin{equation*}
\argmin_{Z \succeq 0} \norm{Z + M}^2,
\end{equation*}
which is the projection of $-M$ onto the cone of positive semidefinite matrices.
It is well known that the solution  $Z = \projpsd(-M) = -\projnsd(M)$ can be computed
from the eigenvalue decomposition of $M$,  see \cite{higham1988computing}. 

By eliminating $Z$ in~\eqref{eq:inner_problem} we get
\begin{align}
\label{eq:inner_problem_red}
\begin{split}
\max \quad&  b^\top \dualvarone - 
			\frac{1}{2\alpha}
				\norm[\Big]{\projpsd\roundbr[\big]{W^\top\roundbr[\big]{\mathcal A^\top \dualvarone - \mathcal B^\top\dualvartwo
					+ S - \Lxl}W + \alpha R} }^2 + \frac{\alpha}{2}\norm{R}^2\\
\textrm{s.t.} \quad& \dualvarone \in \mathbb{R}^{p}, \dualvartwo \geq 0, \ S \ge 0.
\end{split}
\end{align}
Let $F_{\alpha}(\dualvarone, \dualvartwo, S)$ 
denote the objective function of~\eqref{eq:inner_problem_red}.
We can then state the following properties of the objective function.
\begin{proposition}
Let $\alpha > 0$ and $F_{\alpha}(\dualvarone, \dualvartwo, S)$ be the objective function of~\eqref{eq:inner_problem_red}.
The function~$F_{\alpha}(\dualvarone, \dualvartwo, S)$ is concave and
differentiable,  with partial gradients given by
\begin{align*}
	\nabla_S F_{\alpha}(\dualvarone,\dualvartwo,S) &= -\frac{1}{\alpha}W \projpsd\roundbr[\big]{W^\top 
										\roundbr[\big]{\mathcal A^\top \dualvarone - \mathcal B^\top\dualvartwo + S - \Lxl}W + \alpha R} W^\top,\\
\nabla_{\dualvarone_i} F_{\alpha}(\dualvarone,\dualvartwo,S) &= b_i - 
		\frac{1}{\alpha}\inprod[\Big]{A_i, W \projpsd\roundbr[\Big]{W^\top\roundbr[\big]{\mathcal A^\top \dualvarone -\mathcal B^\top\dualvartwo + S - \Lxl}W + \alpha R }W^\top}\\
\nabla_{\dualvartwo_j} F_{\alpha}(\dualvarone,\dualvartwo,S) &= \frac{1}{\alpha}\inprod[\Big]{B_j, W \projpsd\roundbr[\Big]{W^\top\roundbr[\big]{\mathcal A^\top \dualvarone -\mathcal B^\top\dualvartwo + S - \Lxl}W + \alpha R }W^\top}
\end{align*}
for all $1 \leq i \leq p$, and $1 \leq j \leq q$.
\end{proposition}
\begin{proof}
Let $f \colon \R^{n+1} \to \R$ be defined by $f(x) = \frac{1}{2}\Vert x
_+ \Vert^2$, where $x_+$ is the vector of length~$n+1$ with $\roundbr[\big]{x_+}_i = \max (0, x_i)$. Then $f$ is a convex and
differentiable function with gradient $\nabla f(x) = x_+$.
Note that the objective function $F_{\alpha}$ can be written as 
$$
F_{\alpha}(\dualvarone, \dualvartwo, S) = b^\top \dualvarone - \frac{1}{\alpha}g\left( W^\top 
										\roundbr[\big]{\mathcal A^\top \dualvarone -\mathcal B^\top\dualvartwo + S - \Lxl}W + \alpha R \right) + \frac{\alpha}{2}\norm{R}^2 
$$
where the function $g \colon \symmmat{n+1} \to \R$ is defined by 
$$
g(X) = \frac{1}{2}\norm{\projpsd(X)}^2 = \frac{1}{2}\sum_{i=1}^{n+1}\left( \max(0,\lambda_i(X))\right)^2 = f(\lambda(X)),
$$
where $\lambda(X)$ denotes the vector containing all eigenvalues of $X$.
From \cite[\S5.2]{borwein2005convex} we have that the function $g$ is convex and differentiable with gradient $\nabla g(X) =  \projpsd(X)$.  Hence, we can conclude that $F_{\alpha}$ is concave and differentiable. 

To obtain the gradients, we apply the chain rule.
For this, let~$D_1, D_2, D_3 \in \symmmat{n+1}$ where $D_1$ is independent of~$\dualvarone$, $D_2$ is independent of~$\dualvartwo$
and $D_3$ is independent of~$S$.
Let further 
$\mathcal{M}_1 \colon \R^p \to \symmmat{n+1}$, $\mathcal{M}_2 \colon \R^q \to \symmmat{n+1}$ and $\mathcal{N} \colon \symmmat{2n+3} \to \symmmat{n+1}$ be
linear operators defined by 
\begin{align*}
\mathcal{M}_1\dualvarone &= W^\top(A^\top \dualvarone)W = \sum_{i=1}^{p} \dualvarone_i\left(W^\top A_iW\right),\\
\mathcal{M}_2\dualvartwo &= W^\top(B^\top \dualvartwo)W = \sum_{j=1}^{q} \dualvartwo_j\left(W^\top B_jW\right) \text{ and}\\
\mathcal{N}(S) &= W^\top S W.
\end{align*}
Their adjoints are $\left(\mathcal{M}_1^*X \right)_i = \langle W^\top A_i W, X\rangle$, $(\mathcal M_2^*X)_j = \inprod{W^\top B_j W, X}$  and 
$\mathcal{N}^*(X) = W X W^\top$, respectively. 
Now, applying the chain rule, we get
\begin{align*}
	\nabla_\dualvarone\left[g\left(\mathcal{M}_1\dualvarone + D_1\right)\right] &= \mathcal{M}_1^*\nabla g(\mathcal{M}_1\dualvarone + D_1), \\
	\nabla_\dualvartwo\left[g\left(-\mathcal{M}_2\dualvartwo + D_2\right)\right] &= -\mathcal{M}_2^*\nabla g(-\mathcal{M}_2\dualvartwo + D_2) \text{ and}\\
	\nabla_S\left[g(\mathcal{N}(S)+D_3)\right] &= \mathcal{N}^*\nabla g(\mathcal{N}(S)+D_3),
\end{align*}
verifying our stated gradients.
\end{proof}
Using this proposition allows that for a fixed $\alpha > 0$ and $R$ the function $-F_{\alpha}$
is minimized using the L-BFGS-B algorithm~\cite{byrd1995limited}.

To make the algorithm efficient in practice, for a fixed set of cuts~$\mathcal T$ and $\alpha$, we perform only one iteration
of the augmented Lagrangian method,  i.e.,  we approximately solve the inner
problem~\eqref{eq:inner_problem_red},  update $R$  by~\eqref{eq:update_primal_matrix},  and proceed to search for
new violated inequalities in~$WRW^\top$ and add them to~$\mathcal T$.
If we find no or only a few violations, we reduce $\alpha$ and repeat
until the penalty parameter~$\alpha$ is smaller than a certain threshold.
At the end, we perform several augmented Lagrangian iterations with the same~$\alpha$
and without adding new cuts.

A similar solution technique was also used in~\cite{krislock2014improved}, where the penalty method is applied to the dual problem 
and the resulting nonlinear function is minimized using a quasi-Newton method, and in~\cite{hrga2023solving}, where the augmented 
Lagrangian method is used to solve an SDP relaxation strengthened by cutting planes. We have extended this methodology to facially 
reduced DNN programs.

\subsection{Post processing to derive a valid lower bound}
Our approach yields a dual solution $(\dualvarone, \dualvartwo, S, Z)$ of moderate precision. In particular,
the constraint~$W^\top \roundbr{\mathcal{A}^\top \dualvarone - \mathcal B^\top\dualvartwo + S - \Lxl}W + Z = 0$ does not necessarily hold.
However, because of weak duality, every feasible solution of the dual~\eqref{eq:dual_sdp-relax-facialred} gives a lower bound
on the optimal value of~\eqref{eq:DNN-CUTS-P-facialred}.
To derive a safe dual lower bound, we adapt the methods of~\cite{Cerulli2021} and~\cite{Jansson2007}, where the post-processing
of~\cite{Jansson2007} is for SDPs and was adapted by~\cite{Cerulli2021} for DNNs.
Both approaches are based on the following Lemma.
\begin{lemma}[Lemma 3.1 in~\cite{Jansson2007}]\label{lemma:jansson}
	Let $A, B \in \symmmat{2n+3}$ with
	\begin{equation*}
		\underline{b} \leq \lambda_{\min}(B), \quad 0 \leq \lambda_{\min}(A), \quad \lambda_{\max}(A) \leq \overline{a}
	\end{equation*}
	for some $\underline{b}, \overline{a} \in \R$. Then the inequality
	\[ \inprod{A,B} \geq \overline{a} \sum_{i:\lambda_{i}(B) < 0} \lambda_{i}(B) \geq \overline{a}\roundbr{2n+3}\min\set{0,\underline{b}} \]
	holds.
\end{lemma}
This leads us to the following theorem of~\cite{Jansson2007,Cerulli2021} adapted for facially reduced DNNs.
\begin{theorem}\label{thm:postprocessing}
	Consider the facially reduced primal problem~\eqref{eq:DNN-CUTS-P-facialred}, let $R^*$ be an optimal solution
	and let $p^*$ be its optimal value. Given $\dualvarone \in \R^p$, $\dualvartwo \in \R^q$ and $S \in \symmmat{2n+3}$
	with $\dualvartwo \geq 0$ and $S \geq 0$, set
	\begin{equation*}
		\widetilde{Z} = W^\top\roundbr[\big]{\Lxl - \mathcal{A}^\top \dualvarone + \mathcal B^\top \dualvartwo - S}W
	\end{equation*}
	and suppose that $\underline{z} \leq \lambda_{\min}(W\widetilde{Z}W^\top)$.
	Assume $\overline{r} \in \R$ such that $\overline{r} \geq \lambda_{\max}(WR^*W^\top)$ is known.
	Then the inequality
	\begin{equation}\label{eq:safe-dualbound}
		p^* \geq b^\top \dualvarone + \overline{r} \sum_{i: \lambda_i\roundbr[\big]{W\widetilde{Z}W^\top} < 0} \lambda_i(W\widetilde{Z}W^\top) \geq 
					b^\top \dualvarone + \overline{r}\roundbr{2n + 3}\min\set{0,\underline{z}}
	\end{equation}
	holds.
\end{theorem}
\begin{proof}
	Let $R^*$ be optimal for~\eqref{eq:DNN-CUTS-P-facialred}, then
	\begin{align*}
		\inprod[\big]{W^\top \Lxl W,R^*} - b^\top \dualvarone &= \inprod[\big]{\Lxl, WR^*W^\top} - \inprod[\big]{\mathcal{A}\roundbr{WR^*W^\top}, \dualvarone}\\
				&= \inprod[\big]{\Lxl - \mathcal{A}^\top \dualvarone, WR^*W^\top} \\
				&= \inprod[\big]{W^\top\roundbr[\big]{\Lxl - \mathcal{A}^\top \dualvarone}W, R^*} \\
				&= \inprod[\big]{W^\top (S - \mathcal B^\top\dualvartwo)W + \widetilde{Z}, R^*}\\
				&= \inprod[\big]{S, WR^*W^\top} - \inprod[\big]{\mathcal B^\top \dualvartwo, WR^*W^\top} + \inprod[\big]{\widetilde{Z},R^*}\\
				&\geq - \inprod[\big]{\mathcal B^\top \dualvartwo, WR^*W^\top} + \inprod[\big]{\widetilde{Z}, R^*}\\
				&= -\dualvartwo^\top \mathcal B\roundbr[\big]{WR^*W^\top} + \inprod[\big]{\widetilde{Z}, R^*} \\
				&\geq \inprod[\big]{\widetilde{Z}, R^*}
	\end{align*}
	holds, where the first inequality follows from the fact that $S \geq 0$ and $WR^*W^\top \geq 0$, and the second inequality
	follows from the fact that $\dualvartwo \geq 0$ and $\mathcal B\roundbr[\big]{WR^*W^\top} \leq 0$.
	Consequently, with \Cref{lemma:jansson} the inequality
	\begin{align*}
		p^* &= \inprod[\big]{W^\top \Lxl W, R^*} \\
			&\geq b^\top \dualvarone + \inprod[\big]{\widetilde{Z}, R^*}\\
			&= b^\top \dualvarone + \inprod[\big]{W\widetilde{Z}W^\top, WR^*W^\top}\\
			& \geq b^\top \dualvarone + \overline{r} \sum_{i: \lambda_i\roundbr[\big]{W\widetilde{Z}W^\top} < 0} \lambda_i(W\widetilde{Z}W^\top) \geq 
			b^\top \dualvarone + \overline{r}\roundbr{2n + 3}\min\set{0,\underline{z}}
	\end{align*}
	holds.
\end{proof}

Given that we know an upper bound on the largest eigenvalue of $WR^*W^\top$, where~$R^*$ is an optimal solution
of~\eqref{eq:DNN-CUTS-P-facialred} and hence also feasible for~\eqref{eq:sdp-relax-facialred},
we can compute a safe lower bound on~\eqref{eq:DNN-CUTS-P-facialred} by applying \Cref{thm:postprocessing} 
in a post-processing step.
To this end, we state an upper bound on the largest eigenvalue of all feasible solutions~$\Yxl$ of~\eqref{eq:sdp-relax-facialred}
in the next Proposition.
Again, note that $WR^*W^\top = \Yxl^*$ is optimal for~\eqref{eq:sdp-relax-2np3} under the
assumption that $R^*$ is optimal for~\eqref{eq:sdp-relax-facialred}.
\begin{proposition}\label{prop:ubtrY}
	Let $\Yxl = WRW^\top$ be a feasible solution of~\eqref{eq:sdp-relax-facialred}, then
	\[\lambda_{\max}(\Yxl) \leq \tr(\Yxl) \leq \Big\lfloor \frac n 2 \Big\rfloor^2 + n\]
	holds.
\end{proposition}
\begin{proof}
	Since the matrix $\Yxl$ is positive semidefinite, its eigenvalues are non-negative and 
	therefore the trace gives an upper bound on the largest eigenvalue. 
	
	From \Cref{prop:bounds-rho,prop:ub-entries-Y} we already know that
	$\rho \leq 1$ and $\Yb^{33} + \Yb^{44} \leq \big\lfloor \frac n 2 \big\rfloor^2 - 1$ holds for every feasible solution.
	For the submatrices $\Yb^{11}$ and $\Yb^{22}$ of $\Yxl$, consider for~$i \in [n]$ the equations
	\begin{align*}
		0 &= (M\Yxl M^\top)_{2+i,2+i} = \Yb_{ii}^{11} + \Yb_{ii}^{22} + \rho + 2 \Yb_{ii}^{12} - 2\yb_i^1 - 2\yb_i^2.
	\end{align*}
	Since $\Yb_{ii}^{12} = 0$, this yields the equality $\Yb_{ii}^{11} + \Yb_{ii}^{22} = 2 (\yb_i^1 + \yb_i^2) - \rho$.
	Summing over all indices~$i \in [n]$ gives
	\begin{align*}
	\tr(\Yb^{11}) + \tr(\Yb^{22}) &= \sum_{i = 1}^n (\Yb_{ii}^{11} + \Yb_{ii}^{22}) = 2 (e_n^\top \yb^1 + e_n^\top \yb^2) - n\rho\\
	&= 2 \big(e_n^\top \yb^1 + e_n^\top(\rho e_n - \yb^1)\big) - n\rho = n\rho, 
	\end{align*}
	where we used the equality $\yb^2 = \rho e_n - \yb^1$ from~\Cref{prop:bounds-rho}.
	To sum up, we obtain
	\begin{align*}
		\tr(\Yxl) &= \tr(\Yb^{11}) + \tr(\Yb^{22}) + \Yb^{33} + \Yb^{44} + \rho \\
				  &\leq (n + 1) \rho + \Big\lfloor \frac n 2 \Big\rfloor^2 - 1 \leq \Big\lfloor \frac n 2 \Big\rfloor^2 + n,
	\end{align*}
	which is the claimed upper bound.
\end{proof}

We summarize our bounding routine in Algorithm~\ref{alg:bounding_routine}. Note that separating the BQP inequalities \eqref{subeq:bqpb} can be done in $\mathcal{O}(n^3)$ by complete enumeration. 
\begin{algorithm}
\small
\caption{Computation of DNN bound with BQP inequalities}
\label{alg:bounding_routine}
\begin{algorithmic}
\State{\textbf{input:}} $\alpha > 0$, matrices $\Lxl$,  $W$, initial $\Yxl$, $\mathcal{T} = \emptyset$,
$n_{bqp}$ - maximum number of BQP inequalities to be added in one iteration.
\While{$\alpha$ is sufficiently large}
\State Get an approximate maximizer  of \eqref{eq:inner_problem_red} by using L-BFGS-B method.
\State Update $R$ and $\Yxl$ according to \eqref{eq:update_primal_matrix} and \eqref{eq:construct_primal_matrix}.
\State Remove inactive BQP inequalities from $\mathcal{T}$.
\State Add the $n_{bqp}$ most violated BQP inequalities for $\Yxl$ to $\mathcal{T}$.
\If{the number of new violated inequalities added is below a threshold}
    \State Reduce $\alpha$.
\EndIf
\EndWhile
\State Perform extra augmented Lagrangian iterations (with the previous $\alpha$, without adding new cuts).
\State Compute $\widetilde{Z} = W^\top\roundbr[\big]{\Lxl - \mathcal{A}^\top \dualvarone + \mathcal B^\top \dualvartwo - S}W$.
\State Use \eqref{eq:safe-dualbound} with $\bar{r} = \big\lfloor\frac n 2\big\rfloor^2 + n$ to compute a valid lower bound LB.
\State{\textbf{output:}} LB
\end{algorithmic}
\end{algorithm}

\section{Numerical results}\label{sec:numericalresults}
We implemented Algorithm~\ref{alg:bounding_routine} in Julia~\cite{bezanson2017julia} version 1.10.0.
For computing the basic DNN bound~\eqref{eq:sdp-relax-np1} we are using Mosek~v10.2.3~\cite{mosek10.2.3julia} with the
modeling language JuMP~\cite{Lubin2023}.
We use the Julia package LBFGSB~v0.4.1~\cite{LBFGSBjulia} as a wrapper to the \mbox{L-BFGS-B} solver.
All computations were carried out on an AMD EPYC 7343 with 16~cores 
with~4.00~GHz and 1024GB~RAM, operated under Debian GNU/Linux~11.
The code is available at the arXiv page of this paper and at~\url{https://github.com/melaniesi/CheegerConvexificationBounds.jl}.

The initial value of $\alpha$ is set to 1, and we stop reducing it once the penalty parameter~$\alpha$ is smaller than $10^{-5}$.
We start to add violated cuts after the first five iterations.
In each iteration, we then add the~$500$ most violated cuts, where we consider a BQP inequality
to be violated if the violation is at least~$10^{-3}$.
In case we find less than~50 new violated cuts, we reduce $\alpha$ by a factor of~$\frac{3}{5}$.
Cuts are removed from~$\mathcal T$ if the corresponding dual value is smaller than~$10^{-5}$.
The cuts are added for the upper left~$n \times n$ submatrix~$\Yb^{11}$ only.
We perform at most~500 augmented Lagrangian iterations additionally
at the end with constant~$\alpha$ and stop as soon as
the correction of the post processing is smaller than~$0.01$.
We set the parameters of the L-BFGS-B solver to \texttt{maxiter=2000}, \texttt{factr=1e8}, and
\mbox{\texttt{m = 10}}.

As benchmark instances, we use all instances already
considered in~\cite{gupte2024journal} with less than~400~vertices
and larger grlex and grevlex instances, cf.~\cite{gupte2019dantzigfigures}. 
In~\cite{gupte2024journal} only grlex and grevlex instances with up to~92~vertices
were considered, whereas we include all instances with fewer than~400~vertices.
The benchmark set consists of graphs of random~0/1-polytopes, graphs of
grlex and grevlex polytopes, instances from the 10\textsuperscript{th}~DIMACS challenge and some
network instances.
The edge expansion was computed in~\cite{gupte2024journal} for all instances considered therein.
It was shown in~\cite{gupte2019dantzigfigures} that the edge expansion of all
grlex instances is~1.
Upper bounds, conjectured to be the edge expansion of the larger grevlex instances
are taken from~\cite{Rescher2023}.

\subsection{Comparing \eqref{eq:sdp-relax-np1} with~\eqref{eq:sdp-relax-2np3}}
In \Cref{sec:simple} we proved that our new relaxation~\eqref{eq:sdp-relax-2np3} is at least as good as the
basic relaxation~\eqref{eq:sdp-relax-np1}. We now compute these two bounds
for several instances in order to show the significant dominance of
our new relaxation.
Table~\ref{tab:comp-dnnsmall} is structured as follows.
In the first three columns, we list the name of the instance, the number of vertices~$n$, and
the number of edges~$m$. The edge expansion, or if not known, an upper bound on it, is reported in the fourth
column~UB of the table.
In column~5 we report the lower bound~\eqref{eq:sdp-relax-np1} computed with Mosek, and in column~6 the
relative gap of this lower bound to the upper bound.
The relative gap is computed with the formula
\begin{equation}\label{eq:gapformula}
	\frac{\text{UB} - \text{LB}}{\text{UB}},
\end{equation}
where LB denotes a lower bound on the edge expansion, and UB denotes the upper bound on the
edge expansion.
In the last two columns of Table~\ref{tab:comp-dnnsmall} we report the bound~\eqref{eq:sdp-relax-2np3}
approximated by~\Cref{alg:bounding_routine} without BQP inequalities and without the diagonal constraint~\eqref{eq:diagonalcon}
and its relative gap to the upper bound.

Note that all reported bounds are rounded to two decimal places.

Table~\ref{tab:comp-dnnsmall} shows that except for the instance malariagenes-HVR1,
the DNN relaxation~\eqref{eq:sdp-relax-2np3} of
dimension~$2n+3$ drastically reduces the relative gap to the upper bound compared to the
basic DNN relaxation~\eqref{eq:sdp-relax-np1} of dimension~$n+1$.
The relative gap of~\eqref{eq:sdp-relax-np1}
is between~21\% and~72\% for all of our instances, with an average relative gap of~46.9\%.
The average relative gap of~\eqref{eq:sdp-relax-2np3} is~8.6\%.
For~19~instances, the relative gap of~\eqref{eq:sdp-relax-2np3} is at most~1\%.
Especially for most of the grlex instances, we were able to obtain a lower bound that is
close to the optimum.
For~9~instances, this bound rounded to two decimal places is equal to the rounded
optimal values.
Still, there are several benchmark instances with a relative gap of~\eqref{eq:sdp-relax-2np3}
being above~10\%.

\begingroup
	\setlength{\tabcolsep}{7pt}
	\begin{longtable}{lrrr|rrrr}
		\caption{Comparison of DNN relaxations without additional constraints.}\label{tab:comp-dnnsmall}\\
		\toprule
		\multicolumn{4}{c}{Instance} & \multicolumn{2}{c}{\eqref{eq:sdp-relax-np1}} & \multicolumn{2}{c}{\eqref{eq:sdp-relax-2np3}} \\
		\cmidrule(r){1-4} \cmidrule(lr){5-6} \cmidrule(l){7-8}
		 & \multicolumn{1}{c}{$n$} & \multicolumn{1}{c}{$m$} & \multicolumn{1}{l}{UB}
		 & \multicolumn{1}{l}{bound} & \multicolumn{1}{l}{gap (\%)} & 
		 \multicolumn{1}{l}{bound} & \multicolumn{1}{l}{gap (\%)} \\\midrule \endfirsthead
		 \caption*{Table~\ref{tab:comp-dnnsmall} (cont.): Comparison of DNN relaxations without additional constraints.}\\
		 \toprule
		 \multicolumn{4}{c}{Instance} & \multicolumn{2}{c}{\eqref{eq:sdp-relax-np1}} & \multicolumn{2}{c}{\eqref{eq:sdp-relax-2np3}} \\
		 \cmidrule(r){1-4} \cmidrule(lr){5-6} \cmidrule(l){7-8}
		 & \multicolumn{1}{c}{$n$} & \multicolumn{1}{c}{$m$} & \multicolumn{1}{l}{UB}
		& \multicolumn{1}{l}{bound} & \multicolumn{1}{l}{gap (\%)} & 
		\multicolumn{1}{l}{bound} & \multicolumn{1}{l}{gap (\%)} \\\midrule\endhead
		\bottomrule \endlastfoot
		\begin{filecontents*}{results-randpolytope.csv}
instance,n,m,ubopt,foptlast,btnu,lb,gap,relgapclosed,time,iterations,ncuts,iterationsextra,foptlastdnn,btnudnn,lbdnn,gapdnn,timednn,iterationsdnn,iterationsextradnn,lbdnnsmall,gapdnnsmall,timednnsmall
rand01-9-153-0,153,4081,18.75,18.68,18.68,18.67,0.4,92.1,136.6,34,1117,0,17.76,17.76,17.75,5.3,67.0,27,4,14.33,23.6,84.6
rand01-9-153-1,153,4044,18.49,18.43,18.43,18.43,0.3,93.4,105.9,30,907,0,17.58,17.57,17.57,5.0,68.9,24,1,14.48,21.7,84.1
rand01-9-153-2,153,4107,19.00,18.95,18.95,18.94,0.3,94.5,167.5,63,1406,30,17.84,17.84,17.83,6.2,77.3,36,13,14.26,25.0,83.7
rand01-8-164-0,164,1868,5.77,5.75,5.74,5.74,0.5,96.9,283.6,78,2073,20,4.85,4.85,4.85,15.9,114.7,35,12,3.66,36.5,104.2
rand01-8-164-1,164,1837,5.35,5.34,5.34,5.33,0.4,97.0,211.1,56,1510,11,4.69,4.69,4.68,12.5,162.7,76,53,3.28,38.8,127.5
rand01-8-164-2,164,1808,5.74,5.71,5.71,5.70,0.8,95.6,521.3,57,2782,0,4.71,4.71,4.71,18.1,216.0,129,106,2.82,51.0,106.5
rand01-9-178-0,178,4590,17.08,16.98,16.97,16.97,0.6,89.6,165.0,43,1978,0,16.08,16.07,16.07,5.9,96.8,33,10,13.33,22.0,142.1
rand01-9-178-1,178,4467,16.71,16.61,16.61,16.60,0.7,91.8,395.0,70,2174,19,15.39,15.38,15.37,8.0,121.8,36,13,11.35,32.1,124.6
rand01-9-178-2,178,4537,16.75,16.68,16.68,16.67,0.5,93.0,223.8,62,1737,16,15.64,15.64,15.64,6.7,112.5,30,7,10.67,36.3,150.9
rand01-8-189-0,189,1768,4.22,4.22,4.21,4.21,0.2,98.9,482.6,48,2278,4,3.47,3.46,3.46,18.2,163.4,27,4,2.61,38.2,182.1
rand01-8-189-1,189,1745,4.04,4.03,4.03,4.03,0.3,98.1,593.5,91,1963,34,3.37,3.37,3.36,16.9,295.1,242,219,2.47,38.8,211.3
rand01-8-189-2,189,1719,4.06,4.06,4.06,4.05,0.2,98.7,792.6,73,2723,16,3.35,3.35,3.33,17.9,291.1,523,500,2.28,43.8,214.1
rand01-9-203-0,203,4900,15.14,15.10,15.09,15.09,0.3,95.5,304.7,100,1787,65,14.01,14.01,14.00,7.5,167.1,46,23,11.11,26.6,236.3
rand01-9-203-1,203,4781,14.84,14.82,14.81,14.81,0.2,97.6,456.1,51,2133,6,13.55,13.55,13.54,8.8,175.4,100,77,9.87,33.5,201.9
rand01-9-203-2,203,4720,14.38,14.35,14.34,14.34,0.2,96.4,348.6,50,1770,4,13.38,13.38,13.38,7.0,176.6,27,4,8.86,38.4,210.7
rand01-9-228-0,228,5065,13.24,13.17,13.16,13.16,0.5,94.1,637.2,41,2491,0,12.03,12.03,12.02,9.2,269.4,61,38,8.94,32.4,433.4
rand01-9-228-1,228,4927,9.00,9.51,12.82,8.93,0.7,43.3,1111.7,545,2972,500,9.95,11.52,8.88,1.3,524.4,523,500,4.35,51.7,332.6
rand01-9-228-2,228,4984,12.82,12.75,12.74,12.74,0.6,92.4,529.4,45,1660,4,11.79,11.77,11.77,8.2,505.3,184,161,7.12,44.5,429.0
rand01-9-253-0,253,5258,11.87,11.78,11.77,11.77,0.8,91.8,1364.5,61,3234,7,10.69,10.67,10.67,10.1,797.6,345,322,7.25,39.0,684.2
rand01-9-253-1,253,5053,9.00,10.45,11.23,8.94,0.6,46.9,1283.1,547,3124,500,9.74,9.95,8.89,1.2,656.4,523,500,4.34,51.8,508.7
rand01-9-253-2,253,5072,11.22,11.19,11.18,11.18,0.4,96.2,915.9,50,2609,10,10.11,10.10,10.09,10.1,536.6,232,209,6.68,40.5,622.1
rand01-10-256-0,256,11056,30.48,30.25,30.25,30.24,0.8,78.6,486.6,41,1020,8,29.41,29.38,29.38,3.6,302.0,31,8,20.20,33.7,739.2
rand01-10-256-1,256,10611,28.84,28.82,28.82,28.82,0.1,98.0,719.0,42,2122,0,27.70,27.69,27.69,4.0,775.0,470,447,17.22,40.3,850.3
rand01-10-256-2,256,10746,29.38,29.20,29.17,29.17,0.7,83.6,469.5,49,2128,7,28.14,28.13,28.13,4.2,318.6,37,14,22.28,24.1,771.1
rand01-9-278-0,278,5224,10.07,9.79,9.79,9.79,2.8,76.1,760.7,49,2204,2,8.91,8.89,8.89,11.7,878.5,328,305,5.99,40.5,995.2
rand01-9-278-1,278,5007,9.00,9.43,9.60,8.30,7.8,26.7,2922.7,580,4272,500,8.26,8.28,8.05,10.6,1180.1,523,500,4.14,54.0,924.9
rand01-9-278-2,278,5132,9.92,9.77,9.77,9.76,1.6,87.3,1737.8,61,3091,4,8.68,8.66,8.65,12.8,763.3,276,253,6.11,38.4,981.2
rand01-10-281-0,281,11828,28.90,28.67,28.67,28.66,0.8,80.0,680.6,52,2068,12,27.73,27.71,27.71,4.1,408.3,40,17,20.21,30.1,1114.5
rand01-10-281-1,281,11490,27.79,27.67,27.66,27.66,0.5,89.7,858.7,48,2233,7,26.46,26.46,26.46,4.8,1103.3,406,383,15.11,45.6,1122.0
rand01-10-281-2,281,11454,27.75,27.56,27.54,27.54,0.8,83.7,693.2,38,1997,2,26.45,26.44,26.44,4.7,418.1,29,6,19.31,30.4,1113.9
\end{filecontents*}
\csvreader[head to column names, late after line = \\]{results-randpolytope.csv}{n=\size,m=\edges}
		{\instance & \size & \edges & \ubopt & \lbdnnsmall & \gapdnnsmall & \lbdnn &  \gapdnn}
		\vspace*{-0.25cm}& & & & & & &\\ 
		\begin{filecontents*}{results-grlex.csv}
instance,n,m,ubopt,foptlast,btnu,lb,gap,relgapclosed,time,iterations,ncuts,iterationsextra,foptlastdnn,btnudnn,lbdnn,gapdnn,timednn,iterationsdnn,iterationsextradnn,lbdnnsmall,gapdnnsmall,timednnsmall
%grlex-2,4,4,1.00,-,-,-,-,-,-,-,-,-,1.00,1.00,1.00,0.0,0.0,23,0,1.00,0.0,0.0
%grlex-3,7,11,1.00,-,-,-,-,-,-,-,-,-,1.00,1.00,1.00,0.0,0.0,23,0,1.00,0.0,0.0
%grlex-4,11,24,1.00,-,-,-,-,-,-,-,-,-,1.00,1.00,1.00,0.0,0.1,23,0,0.78,21.8,0.0
%grlex-5,16,45,1.00,1.00,1.00,1.00,0.2,79.0,2.2,32,677,0,0.99,0.99,0.99,1.1,0.8,23,0,0.60,40.2,0.0
%grlex-6,22,76,1.00,1.00,1.00,1.00,0.0,100.0,0.6,24,194,0,0.96,0.96,0.95,4.5,1.7,23,0,0.53,46.6,0.0
grlex-7,29,119,1.00,1.00,1.00,1.00,0.0,99.4,1.4,27,316,0,0.99,0.99,0.98,1.7,2.9,23,0,0.51,49.3,0.1
grlex-8,37,176,1.00,-,-,-,-,-,-,-,-,-,1.00,1.00,1.00,0.1,3.2,23,0,0.47,53.0,0.2
grlex-9,46,249,1.00,-,-,-,-,-,-,-,-,-,1.00,1.00,1.00,0.2,5.4,23,0,0.43,56.6,0.4
grlex-10,56,340,1.00,-,-,-,-,-,-,-,-,-,1.00,1.00,1.00,0.5,9.3,24,1,0.41,58.5,0.6
grlex-11,67,451,1.00,-,-,-,-,-,-,-,-,-,1.00,1.00,0.99,0.9,12.0,30,7,0.41,59.5,2.2
grlex-12,79,584,1.00,1.00,1.00,1.00,0.3,60.4,27.6,47,1423,20,1.00,1.00,0.99,0.8,26.2,37,14,0.39,60.8,3.1
grlex-13,92,741,1.00,1.00,1.00,1.00,0.3,72.7,47.6,60,885,34,1.00,1.00,0.99,1.0,43.2,60,37,0.38,62.3,7.4
grlex-14,106,924,1.00,-,-,-,-,-,-,-,-,-,1.00,1.00,1.00,0.3,69.1,100,77,0.37,63.2,9.9
grlex-15,121,1135,1.00,-,-,-,-,-,-,-,-,-,1.00,1.00,1.00,0.0,88.9,61,38,0.36,63.6,28.8
grlex-16,137,1376,1.00,1.00,1.00,1.00,0.3,69.9,142.1,70,774,43,1.00,1.00,0.99,0.9,119.8,104,81,0.36,64.3,43.2
grlex-17,154,1649,1.00,-,-,-,-,-,-,-,-,-,1.00,1.00,1.00,0.1,178.9,76,53,0.35,65.1,68.6
grlex-18,172,1956,1.00,-,-,-,-,-,-,-,-,-,1.00,1.00,0.99,0.8,292.3,424,401,0.34,65.6,116.5
grlex-19,191,2299,1.00,-,-,-,-,-,-,-,-,-,1.00,1.00,0.99,0.8,335.1,138,115,0.34,65.9,176.7
grlex-20,211,2680,1.00,1.00,1.00,0.99,1.1,41.4,570.3,526,974,500,1.00,1.00,0.98,1.8,575.5,523,500,0.34,66.3,267.5
grlex-21,232,3101,1.00,-,-,-,-,-,-,-,-,-,1.00,1.00,0.99,0.9,689.2,158,135,0.33,66.8,423.4
grlex-22,254,3564,1.00,-,-,-,-,-,-,-,-,-,1.00,1.00,1.00,0.0,990.7,175,152,0.33,67.1,561.1
grlex-23,277,4071,1.00,1.00,1.00,0.99,1.2,48.4,1455.7,527,1278,500,1.00,1.00,0.98,2.3,1415.0,523,500,0.33,67.3,934.7
grlex-24,301,4624,1.00,1.00,1.00,1.00,0.0,99.2,1663.3,308,968,281,1.00,1.00,0.99,1.0,1542.0,251,228,0.32,67.5,1393.7
grlex-25,326,5225,1.00,1.00,1.00,0.96,3.9,34.0,2250.9,525,443,500,1.00,1.00,0.94,5.9,2231.8,523,500,0.32,67.9,1842.4
grlex-26,352,5876,1.00,-,-,-,-,-,-,-,-,-,1.00,1.00,0.99,0.9,2300.5,258,235,0.32,68.1,2699.8
grlex-27,379,6579,1.00,-,-,-,-,-,-,-,-,-,1.00,1.00,0.96,3.7,2921.6,523,500,0.32,68.2,4113.0
\end{filecontents*}
\csvreader[head to column names, late after line = \\]{results-grlex.csv}{n=\size,m=\edges}
		{\instance & \size & \edges & \ubopt & \lbdnnsmall & \gapdnnsmall & \lbdnn &  \gapdnn}
		\vspace*{-0.25cm}& & & & & & &\\ 
		\begin{filecontents*}{results-grevlex.csv}
instance,n,m,ubopt,foptlast,btnu,lb,gap,relgapclosed,time,iterations,ncuts,iterationsextra,foptlastdnn,btnudnn,lbdnn,gapdnn,timednn,iterationsdnn,iterationsextradnn,lbdnnsmall,lbdnnsmallrounded,gapdnnsmall,timednnsmall
%grevlex-2,4,4,1.00,-,-,-,-,-,-,-,-,-,1.00,1.00,1.00,0.0,0.0,23,0,1.00,1,0.0,0.0
%grevlex-3,7,11,1.33,-,-,-,-,-,-,-,-,-,1.33,1.33,1.33,0.0,0.0,23,0,1.26,1.26,5.3,0.0
%grevlex-4,11,24,1.75,1.57,1.57,1.57,10.6,36.0,0.2,23,58,0,1.46,1.46,1.46,16.5,0.1,23,0,1.32,1.32,24.5,0.0
%grevlex-5,16,45,1.88,1.73,1.73,1.73,7.8,50.5,1.0,31,516,0,1.58,1.58,1.58,15.7,0.2,23,0,1.34,1.34,28.4,0.0
%grevlex-6,22,76,2.00,2.00,2.00,2.00,0.0,100.0,0.6,26,383,0,1.81,1.81,1.81,9.3,0.4,23,0,1.50,1.5,25.0,0.0
grevlex-7,29,119,2.46,2.34,2.34,2.34,4.8,65.7,3.7,40,1613,0,2.12,2.12,2.12,14.0,1.0,23,0,1.73,1.73,29.7,0.1
grevlex-8,37,176,2.83,2.62,2.62,2.62,7.6,55.8,9.4,48,1843,0,2.35,2.35,2.35,17.2,2.2,23,0,1.90,1.9,33.0,0.1
grevlex-9,46,249,2.96,2.85,2.85,2.85,3.6,75.3,9.5,50,1449,0,2.53,2.53,2.52,14.8,4.2,23,0,2.02,2.02,31.6,0.4
grevlex-10,56,340,3.22,3.13,3.13,3.13,2.8,80.3,26.1,66,2136,0,2.78,2.78,2.77,14.1,7.2,28,5,2.21,2.21,31.6,0.7
grevlex-11,67,451,3.67,3.46,3.46,3.46,5.6,65.8,39.0,93,1488,0,3.07,3.07,3.06,16.4,11.2,24,1,2.43,2.43,33.8,1.9
grevlex-12,79,584,3.92,3.75,3.75,3.74,4.5,71.1,53.6,73,2151,0,3.31,3.31,3.31,15.7,24.8,35,12,2.61,2.61,33.4,3.2
grevlex-13,92,741,4.00,4.00,4.00,3.99,0.3,97.4,103.8,77,1652,6,3.52,3.52,3.51,12.3,40.8,34,11,2.77,2.77,30.9,7.3
grevlex-14,106,924,4.47,4.28,4.28,4.28,4.4,72.5,153.6,98,1488,20,3.76,3.76,3.76,16.0,54.4,28,5,2.95,2.95,33.9,11.8
grevlex-15,121,1135,4.83,4.60,4.60,4.59,5.0,69.8,172.7,66,1904,18,4.04,4.03,4.03,16.6,80.5,70,47,3.17,3.17,34.4,29.9
grevlex-16,137,1376,4.99,4.89,4.88,4.88,2.1,85.2,305.2,369,1250,317,4.29,4.29,4.28,14.1,111.7,73,50,3.36,3.36,32.6,52.6
grevlex-17,154,1649,5.27,5.14,5.14,5.13,2.5,82.6,330.8,192,2106,125,4.51,4.50,4.49,14.7,163.2,110,87,3.53,3.53,33.0,83.4
grevlex-18,172,1956,5.67,5.42,5.41,5.40,4.8,70.8,352.8,150,1533,91,4.76,4.76,4.75,16.3,211.6,84,61,3.72,3.72,34.4,125.6
grevlex-19,191,2299,5.94,5.72,5.71,5.71,3.8,75.7,529.0,229,1377,183,5.03,5.03,5.02,15.5,267.5,84,61,3.93,3.93,33.7,208.5
grevlex-20,211,2680,6.06,5.99,5.98,5.98,1.3,90.0,746.1,242,1183,205,5.28,5.26,5.26,13.2,454.4,167,144,4.13,4.13,31.9,333.5
grevlex-21,232,3101,6.53,6.23,6.23,6.23,4.6,70.9,772.6,124,1830,79,5.50,5.50,5.50,15.7,585.1,183,160,4.30,4.3,34.1,485.0
grevlex-22,254,3564,6.83,6.53,6.52,6.51,4.8,70.2,1134.0,178,2240,116,5.75,5.75,5.74,16.0,792.7,188,165,4.49,4.49,34.3,618.4
grevlex-23,277,4071,6.99,6.76,6.74,6.74,3.6,74.5,1310.9,212,1237,176,6.03,6.02,6.02,13.9,956.5,299,276,4.70,4.7,32.7,1099.7
grevlex-24,301,4624,7.35,7.04,7.06,6.70,8.9,45.0,1860.5,544,1558,500,6.28,6.28,6.16,16.3,1836.0,523,500,4.90,4.9,33.4,1601.8
grevlex-25,326,5225,7.72,7.29,7.26,7.28,5.8,69.0,2984.0,463,2317,411,6.50,6.51,6.28,18.7,2544.0,523,500,5.08,5.08,34.3,2117.4
grevlex-26,352,5876,7.94,7.57,7.56,7.51,5.5,73.8,4133.8,557,2373,500,6.75,6.76,6.27,21.0,2169.6,523,500,5.27,5.27,33.7,3393.4
grevlex-27,379,6579,8.13,7.70,7.70,7.64,6.0,69.2,3515.5,537,1450,500,7.02,7.03,6.55,19.4,2747.3,523,500,5.48,5.48,32.6,4405.0
\end{filecontents*}
\csvreader[head to column names, late after line = \\]{results-grevlex.csv}{n=\size,m=\edges}
		{\instance & \size & \edges & \ubopt & \lbdnnsmall & \gapdnnsmall & \lbdnn &  \gapdnn}
		\vspace*{-0.25cm}& & & & & & &\\ 
		\begin{filecontents*}{results-dimacs.csv}
instance,n,m,ubopt,foptlast,btnu,lb,gap,relgapclosed,time,iterations,ncuts,iterationsextra,foptlastdnn,btnudnn,lbdnn,gapdnn,timednn,iterationsdnn,iterationsextradnn,lbdnnsmall,gapdnnsmall,timednnsmall
karate,34,78,0.59,0.59,0.59,0.59,0.0,99.9,5.2,28,518,0,0.55,0.55,0.55,6.7,5.5,23,0,0.24,59.2,0.1
chesapeake,39,170,2.17,-,-,-,-,-,-,-,-,-,2.17,2.17,2.16,0.4,10.3,31,8,0.87,59.8,0.2
dolphins,62,159,0.29,-,-,-,-,-,-,-,-,-,0.29,0.29,0.29,0.1,16.4,23,0,0.09,69.4,1.4
lesmis,77,254,0.30,-,-,-,-,-,-,-,-,-,0.30,0.30,0.30,1.3,40.2,26,3,0.11,64.8,4.9
polbooks,105,441,0.37,0.36,0.36,0.35,3.0,80.0,175.6,40,1033,2,0.32,0.32,0.31,14.9,114.2,146,123,0.18,50.0,13.9
adjnoun,112,425,1.00,-,-,-,-,-,-,-,-,-,1.00,1.00,0.99,0.9,189.8,297,274,0.37,62.9,19.0
football,115,613,1.07,1.07,1.07,1.07,0.0,99.7,118.1,28,382,0,0.98,0.98,0.97,9.4,86.4,28,5,0.76,29.3,21.4
jazz,198,2742,1.00,-,-,-,-,-,-,-,-,-,1.00,1.00,0.99,1.5,806.3,523,500,0.30,70.1,278.6
celegansneural,297,2148,1.00,1.00,1.00,0.96,3.6,56.4,3844.4,541,841,500,1.00,1.00,0.92,8.2,2595.9,523,500,0.43,56.7,1366.6
\end{filecontents*}
\csvreader[head to column names, late after line = \\]{results-dimacs.csv}{n=\size,m=\edges}
		{\instance & \size & \edges & \ubopt & \lbdnnsmall & \gapdnnsmall & \lbdnn &  \gapdnn}
		\vspace*{-0.25cm}& & & & & & &\\ 
		\begin{filecontents*}{results-network.csv}
instance,n,m,ubopt,foptlast,btnu,lb,gap,relgapclosed,time,iterations,ncuts,iterationsextra,foptlastdnn,btnudnn,lbdnn,gapdnn,timednn,iterationsdnn,iterationsextradnn,lbdnnsmall,gapdnnsmall,timednnsmall
moviegalaxies-567,52,146,0.38,-,-,-,-,-,-,-,-,-,0.38,0.38,0.38,1.1,9.8,23,0,0.18,52.7,0.8
moviegalaxies-52,59,119,0.54,0.54,0.54,0.54,0.3,98.4,21.5,27,402,0,0.44,0.44,0.43,19.2,21.0,56,33,0.18,65.7,0.9
terrorists-911,62,152,0.22,-,-,-,-,-,-,-,-,-,0.22,0.22,0.21,3.3,16.3,23,0,0.09,57.3,1.3
train-terrorists,64,243,0.60,0.60,0.60,0.60,0.1,92.7,24.6,28,293,0,0.60,0.60,0.59,1.9,46.7,60,37,0.17,72.0,2.2
highschool,70,274,0.91,0.90,0.90,0.90,1.7,93.5,52.4,61,1140,0,0.68,0.68,0.68,25.8,33.4,30,7,0.41,55.7,2.5
blumenau-drug,75,181,0.50,0.50,0.40,0.49,2.5,62.8,71.4,77,785,33,0.50,0.48,0.47,6.8,46.6,61,38,0.19,62.0,2.5
sp-office,92,9827,3.37,3.30,3.30,3.29,2.4,66.8,89.8,50,1333,2,3.14,3.13,3.13,7.1,64.9,42,19,1.72,48.9,7.0
swingers,96,232,0.33,-,-,-,-,-,-,-,-,-,0.33,0.30,0.32,2.9,123.3,95,72,0.14,58.8,9.1
game-thrones,107,352,0.40,0.40,0.40,0.40,0.4,85.9,108.5,32,279,3,0.40,0.40,0.39,2.7,127.3,28,5,0.12,69.2,12.1
revolution,141,160,0.10,0.10,0.10,0.09,7.9,62.3,233.7,34,1110,1,0.08,0.08,0.08,21.0,242.2,396,373,0.04,59.9,45.0
foodweb,183,2494,1.00,-,-,-,-,-,-,-,-,-,1.19,1.98,0.98,2.1,584.6,523,500,0.50,50.1,139.6
cintestinalis,205,2575,1.00,-,-,-,-,-,-,-,-,-,1.40,1.98,0.98,2.0,833.8,523,500,0.46,53.8,229.0
malariagenes-HVR1,307,2812,0.24,0.24,0.24,0.23,4.4,91.5,2620.7,477,1608,446,0.21,0.21,0.11,52.3,3225.4,523,500,0.11,53.6,2242.2
\end{filecontents*}
\csvreader[head to column names, late after line = \\]{results-network.csv}{n=\size,m=\edges}
		{\instance & \size & \edges & \ubopt & \lbdnnsmall & \gapdnnsmall & \lbdnn &  \gapdnn}
	\end{longtable}
\endgroup

\subsection{Detailed numerical tests of \eqref{eq:DNN-CUTS-P-facialred}}
In the following we now compare~\eqref{eq:sdp-relax-2np3} to relaxation~\eqref{eq:DNN-CUTS-P-facialred}
including additional scaled BQP inequalities, both computed approximately with our \Cref{alg:bounding_routine}.
We omit instances from the table presenting computational results in this subsection
where we could not get an improvement from our cutting planes.
For several of these instances, the bound from~\eqref{eq:sdp-relax-2np3} is already close to the optimal value, hence further improvement from
cutting planes cannot be expected.

Table~\ref{tab:comp-all} is structured as follows.
The first four columns describe the instance, i.e., the name of the instance, the number of vertices~$n$,
the number of edges~$m$, and an upper bound~UB on the edge expansion are given.
In the next~3~columns, we report computational results for computing~\eqref{eq:sdp-relax-2np3}.
Column~5 lists the lower bound on~$h(G)$ obtained by \Cref{alg:bounding_routine}, in column~6
the number of edges~$m$, and an upper bound~UB on the edge expansion are given.
In the next~3~columns, we report computational results for computing~\eqref{eq:sdp-relax-2np3}.
Column~5 lists the lower bound on~$h(G)$ obtained by \Cref{alg:bounding_routine}, column~6 displays
the relative gap~\eqref{eq:gapformula}, and in column~7 the wall-clock time in seconds needed
to compute the lower bound is reported.
The same numbers (bound, relative gap, computation time) for computing~\eqref{eq:DNN-CUTS-P-facialred} 
with \Cref{alg:bounding_routine} are shown in the subsequent~3 columns.
In the second to last column of the table, we report the number of cutting planes left in the last iteration
of the algorithm, i.e., the number of cuts that were added and not removed by the algorithm.
The total number of iterations needed by the algorithm to compute the DNN relaxation strengthened by
scaled BQP cuts is given in the last column of the table.

Table~\ref{tab:comp-all} shows that for all graphs of random 0/1-polytopes, except~3~instances, the lower
bound \eqref{eq:DNN-CUTS-P-facialred} yields a relative gap below~1\%.
The computation time for all instances from graphs of random 0/1-polytopes
was less than 49~minutes and~11.4~minutes on average.
For the grlex instances, we were able to obtain for~5 further instances a lower bound that
is, rounded to two decimal places, equal to the optimum by adding cutting planes to 
the DNN relaxation.
Considering the grevlex instances, there is only one instance where \eqref{eq:DNN-CUTS-P-facialred} yields a
relative gap below~1\% compared to a relative gap of~12.3\% from~\eqref{eq:sdp-relax-2np3}.
Nevertheless, also for this class of benchmark instances, we were able to significantly
reduce the relative gap to the upper bound/optimum by adding BQP inequalities, namely from an
average of~15.8\% to~4.4\%.
A similar observation can be made from the computational results on 
the DIMACS and network instances.
In total, for~10~instances~\eqref{eq:DNN-CUTS-P-facialred} rounded to~2~decimal places
closes the gap to the optimum~$h(G)$.

Overall, our computational results demonstrate
that in case the cutting planes improve the DNN relaxation of the edge expansion,
the improvement is substantial.
We were able to compute strong lower bounds on the edge expansion
by computing \eqref{eq:DNN-CUTS-P-facialred}, a DNN of dimension up to~$761 \times 761$ and up to~$4\,272$
cutting planes, with \Cref{alg:bounding_routine} in less than~69~minutes.

\begin{landscape}
\begingroup
	\setlength{\tabcolsep}{6.2pt}
	\begin{longtable}{lrrr|rrrrrrrr}
		\caption{Computational results of \Cref{alg:bounding_routine}.}\label{tab:comp-all}\\
		\toprule
		\multicolumn{4}{c}{Instance} & \multicolumn{3}{c}{\eqref{eq:sdp-relax-2np3}} & \multicolumn{5}{c}{\eqref{eq:DNN-CUTS-P-facialred}}\\
		\cmidrule(r){1-4} \cmidrule(lr){5-7} \cmidrule(l){8-12}
		 & \multicolumn{1}{c}{$n$} & \multicolumn{1}{c}{$m$} & \multicolumn{1}{l}{UB}
		 & \multicolumn{1}{l}{bound} & \multicolumn{1}{l}{gap (\%)} &  \multicolumn{1}{l}{time (s)} &
		 \multicolumn{1}{l}{bound} & \multicolumn{1}{l}{gap (\%)} & \multicolumn{1}{l}{time (s)} &
		 \multicolumn{1}{l}{cuts} & \multicolumn{1}{l}{iterations}\\\midrule \endfirsthead
		 \caption*{Table~\ref{tab:comp-all} (cont.): Computational results of \Cref{alg:bounding_routine}.}\label{tab:comp-all}\\
		 \toprule
		\multicolumn{4}{c}{Instance} & \multicolumn{3}{c}{\eqref{eq:sdp-relax-2np3}} & \multicolumn{5}{c}{\eqref{eq:DNN-CUTS-P-facialred}}\\
		\cmidrule(r){1-4} \cmidrule(lr){5-7} \cmidrule(l){8-12}
		 & \multicolumn{1}{c}{$n$} & \multicolumn{1}{c}{$m$} & \multicolumn{1}{l}{UB}
		 & \multicolumn{1}{l}{bound} & \multicolumn{1}{l}{gap (\%)} &  \multicolumn{1}{l}{time (s)} &
		 \multicolumn{1}{l}{bound} & \multicolumn{1}{l}{gap (\%)} & \multicolumn{1}{l}{time (s)} &
		 \multicolumn{1}{l}{cuts} & \multicolumn{1}{l}{iterations}\\\midrule \endhead
		\bottomrule \endlastfoot
		\begin{filecontents*}{results2-randpolytope.csv}
instance,n,m,ubopt,foptlast,btnu,lb,gap,relgapclosed,time,iterations,ncuts,iterationsextra,foptlastdnn,btnudnn,lbdnn,gapdnn,timednn,iterationsdnn,iterationsextradnn,lbdnnsmall,gapdnnsmall,timednnsmall
rand01-9-153-0,153,4081,18.75,18.68,18.68,18.67,0.4,92.1,136.6,34,1117,0,17.76,17.76,17.75,5.3,67.0,27,4,14.33,23.6,84.6
rand01-9-153-1,153,4044,18.49,18.43,18.43,18.43,0.3,93.4,105.9,30,907,0,17.58,17.57,17.57,5.0,68.9,24,1,14.48,21.7,84.1
rand01-9-153-2,153,4107,19.00,18.95,18.95,18.94,0.3,94.5,167.5,63,1406,30,17.84,17.84,17.83,6.2,77.3,36,13,14.26,25.0,83.7
rand01-8-164-0,164,1868,5.77,5.75,5.74,5.74,0.5,96.9,283.6,78,2073,20,4.85,4.85,4.85,15.9,114.7,35,12,3.66,36.5,104.2
rand01-8-164-1,164,1837,5.35,5.34,5.34,5.33,0.4,97.0,211.1,56,1510,11,4.69,4.69,4.68,12.5,162.7,76,53,3.28,38.8,127.5
rand01-8-164-2,164,1808,5.74,5.71,5.71,5.70,0.8,95.6,521.3,57,2782,0,4.71,4.71,4.71,18.1,216.0,129,106,2.82,51.0,106.5
rand01-9-178-0,178,4590,17.08,16.98,16.97,16.97,0.6,89.6,165.0,43,1978,0,16.08,16.07,16.07,5.9,96.8,33,10,13.33,22.0,142.1
rand01-9-178-1,178,4467,16.71,16.61,16.61,16.60,0.7,91.8,395.0,70,2174,19,15.39,15.38,15.37,8.0,121.8,36,13,11.35,32.1,124.6
rand01-9-178-2,178,4537,16.75,16.68,16.68,16.67,0.5,93.0,223.8,62,1737,16,15.64,15.64,15.64,6.7,112.5,30,7,10.67,36.3,150.9
rand01-8-189-0,189,1768,4.22,4.22,4.21,4.21,0.2,98.9,482.6,48,2278,4,3.47,3.46,3.46,18.2,163.4,27,4,2.61,38.2,182.1
rand01-8-189-1,189,1745,4.04,4.03,4.03,4.03,0.3,98.1,593.5,91,1963,34,3.37,3.37,3.36,16.9,295.1,242,219,2.47,38.8,211.3
rand01-8-189-2,189,1719,4.06,4.06,4.06,4.05,0.2,98.7,792.6,73,2723,16,3.35,3.35,3.33,17.9,291.1,523,500,2.28,43.8,214.1
rand01-9-203-0,203,4900,15.14,15.10,15.09,15.09,0.3,95.5,304.7,100,1787,65,14.01,14.01,14.00,7.5,167.1,46,23,11.11,26.6,236.3
rand01-9-203-1,203,4781,14.84,14.82,14.81,14.81,0.2,97.6,456.1,51,2133,6,13.55,13.55,13.54,8.8,175.4,100,77,9.87,33.5,201.9
rand01-9-203-2,203,4720,14.38,14.35,14.34,14.34,0.2,96.4,348.6,50,1770,4,13.38,13.38,13.38,7.0,176.6,27,4,8.86,38.4,210.7
rand01-9-228-0,228,5065,13.24,13.17,13.16,13.16,0.5,94.1,637.2,41,2491,0,12.03,12.03,12.02,9.2,269.4,61,38,8.94,32.4,433.4
rand01-9-228-1,228,4927,9.00,9.51,12.82,8.93,0.7,43.3,1111.7,545,2972,500,9.95,11.52,8.88,1.3,524.4,523,500,4.35,51.7,332.6
rand01-9-228-2,228,4984,12.82,12.75,12.74,12.74,0.6,92.4,529.4,45,1660,4,11.79,11.77,11.77,8.2,505.3,184,161,7.12,44.5,429.0
rand01-9-253-0,253,5258,11.87,11.78,11.77,11.77,0.8,91.8,1364.5,61,3234,7,10.69,10.67,10.67,10.1,797.6,345,322,7.25,39.0,684.2
rand01-9-253-1,253,5053,9.00,10.45,11.23,8.94,0.6,46.9,1283.1,547,3124,500,9.74,9.95,8.89,1.2,656.4,523,500,4.34,51.8,508.7
rand01-9-253-2,253,5072,11.22,11.19,11.18,11.18,0.4,96.2,915.9,50,2609,10,10.11,10.10,10.09,10.1,536.6,232,209,6.68,40.5,622.1
rand01-10-256-0,256,11056,30.48,30.25,30.25,30.24,0.8,78.6,486.6,41,1020,8,29.41,29.38,29.38,3.6,302.0,31,8,20.20,33.7,739.2
rand01-10-256-1,256,10611,28.84,28.82,28.82,28.82,0.1,98.0,719.0,42,2122,0,27.70,27.69,27.69,4.0,775.0,470,447,17.22,40.3,850.3
rand01-10-256-2,256,10746,29.38,29.20,29.17,29.17,0.7,83.6,469.5,49,2128,7,28.14,28.13,28.13,4.2,318.6,37,14,22.28,24.1,771.1
rand01-9-278-0,278,5224,10.07,9.79,9.79,9.79,2.8,76.1,760.7,49,2204,2,8.91,8.89,8.89,11.7,878.5,328,305,5.99,40.5,995.2
rand01-9-278-1,278,5007,9.00,9.43,9.60,8.30,7.8,26.7,2922.7,580,4272,500,8.26,8.28,8.05,10.6,1180.1,523,500,4.14,54.0,924.9
rand01-9-278-2,278,5132,9.92,9.77,9.77,9.76,1.6,87.3,1737.8,61,3091,4,8.68,8.66,8.65,12.8,763.3,276,253,6.11,38.4,981.2
rand01-10-281-0,281,11828,28.90,28.67,28.67,28.66,0.8,80.0,680.6,52,2068,12,27.73,27.71,27.71,4.1,408.3,40,17,20.21,30.1,1114.5
rand01-10-281-1,281,11490,27.79,27.67,27.66,27.66,0.5,89.7,858.7,48,2233,7,26.46,26.46,26.46,4.8,1103.3,406,383,15.11,45.6,1122.0
rand01-10-281-2,281,11454,27.75,27.56,27.54,27.54,0.8,83.7,693.2,38,1997,2,26.45,26.44,26.44,4.7,418.1,29,6,19.31,30.4,1113.9
\end{filecontents*}
\csvreader[head to column names, late after line = \\]{results2-randpolytope.csv}{n=\size,m=\edges}
		{\instance & \size & \edges & \ubopt & \lbdnn &  \gapdnn & \timednn & \lb & \gap & \time & \ncuts & \iterations}
		\vspace*{-0.25cm} &  &  &  &  &   &  &  &  &  &  & \\
		\begin{filecontents*}{results2-grlex.csv}
instance,n,m,ubopt,foptlast,btnu,lb,gap,relgapclosed,time,iterations,ncuts,iterationsextra,foptlastdnn,btnudnn,lbdnn,gapdnn,timednn,iterationsdnn,iterationsextradnn,lbdnnsmall,gapdnnsmall,timednnsmall
%grlex-5,16,45,1.00,1.00,1.00,1.00,0.2,79.0,2.2,32,677,0,0.99,0.99,0.99,1.1,0.8,23,0,0.60,40.2,0.0
%grlex-6,22,76,1.00,1.00,1.00,1.00,0.0,100.0,0.6,24,194,0,0.96,0.96,0.95,4.5,1.7,23,0,0.53,46.6,0.0
grlex-7,29,119,1.00,1.00,1.00,1.00,0.0,99.4,1.4,27,316,0,0.99,0.99,0.98,1.7,2.9,23,0,0.51,49.3,0.1
grlex-12,79,584,1.00,1.00,1.00,1.00,0.3,60.4,27.6,47,1423,20,1.00,1.00,0.99,0.8,26.2,37,14,0.39,60.8,3.1
grlex-13,92,741,1.00,1.00,1.00,1.00,0.3,72.7,47.6,60,885,34,1.00,1.00,0.99,1.0,43.2,60,37,0.38,62.3,7.4
grlex-16,137,1376,1.00,1.00,1.00,1.00,0.3,69.9,142.1,70,774,43,1.00,1.00,0.99,0.9,119.8,104,81,0.36,64.3,43.2
grlex-20,211,2680,1.00,1.00,1.00,0.99,1.1,41.4,570.3,526,974,500,1.00,1.00,0.98,1.8,575.5,523,500,0.34,66.3,267.5
grlex-23,277,4071,1.00,1.00,1.00,0.99,1.2,48.4,1455.7,527,1278,500,1.00,1.00,0.98,2.3,1415.0,523,500,0.33,67.3,934.7
grlex-24,301,4624,1.00,1.00,1.00,1.00,0.0,99.2,1663.3,308,968,281,1.00,1.00,0.99,1.0,1542.0,251,228,0.32,67.5,1393.7
grlex-25,326,5225,1.00,1.00,1.00,0.96,3.9,34.0,2250.9,525,443,500,1.00,1.00,0.94,5.9,2231.8,523,500,0.32,67.9,1842.4
\end{filecontents*}
\csvreader[head to column names, late after line = \\]{results2-grlex.csv}{n=\size,m=\edges}
		{\instance & \size & \edges & \ubopt & \lbdnn &  \gapdnn & \timednn & \lb & \gap & \time & \ncuts & \iterations}
		\vspace*{-0.25cm} &  &  &  &  &   &  &  &  &  &  & \\
		\begin{filecontents*}{results2-grevlex.csv}
instance,n,m,ubopt,foptlast,btnu,lb,gap,relgapclosed,time,iterations,ncuts,iterationsextra,foptlastdnn,btnudnn,lbdnn,gapdnn,timednn,iterationsdnn,iterationsextradnn,lbdnnsmall,lbdnnsmallrounded,gapdnnsmall,timednnsmall
%grevlex-4,11,24,1.75,1.57,1.57,1.57,10.6,36.0,0.2,23,58,0,1.46,1.46,1.46,16.5,0.1,23,0,1.32,1.32,24.5,0.0
%grevlex-5,16,45,1.88,1.73,1.73,1.73,7.8,50.5,1.0,31,516,0,1.58,1.58,1.58,15.7,0.2,23,0,1.34,1.34,28.4,0.0
%grevlex-6,22,76,2.00,2.00,2.00,2.00,0.0,100.0,0.6,26,383,0,1.81,1.81,1.81,9.3,0.4,23,0,1.50,1.5,25.0,0.0
grevlex-7,29,119,2.46,2.34,2.34,2.34,4.8,65.7,3.7,40,1613,0,2.12,2.12,2.12,14.0,1.0,23,0,1.73,1.73,29.7,0.1
grevlex-8,37,176,2.83,2.62,2.62,2.62,7.6,55.8,9.4,48,1843,0,2.35,2.35,2.35,17.2,2.2,23,0,1.90,1.9,33.0,0.1
grevlex-9,46,249,2.96,2.85,2.85,2.85,3.6,75.3,9.5,50,1449,0,2.53,2.53,2.52,14.8,4.2,23,0,2.02,2.02,31.6,0.4
grevlex-10,56,340,3.22,3.13,3.13,3.13,2.8,80.3,26.1,66,2136,0,2.78,2.78,2.77,14.1,7.2,28,5,2.21,2.21,31.6,0.7
grevlex-11,67,451,3.67,3.46,3.46,3.46,5.6,65.8,39.0,93,1488,0,3.07,3.07,3.06,16.4,11.2,24,1,2.43,2.43,33.8,1.9
grevlex-12,79,584,3.92,3.75,3.75,3.74,4.5,71.1,53.6,73,2151,0,3.31,3.31,3.31,15.7,24.8,35,12,2.61,2.61,33.4,3.2
grevlex-13,92,741,4.00,4.00,4.00,3.99,0.3,97.4,103.8,77,1652,6,3.52,3.52,3.51,12.3,40.8,34,11,2.77,2.77,30.9,7.3
grevlex-14,106,924,4.47,4.28,4.28,4.28,4.4,72.5,153.6,98,1488,20,3.76,3.76,3.76,16.0,54.4,28,5,2.95,2.95,33.9,11.8
grevlex-15,121,1135,4.83,4.60,4.60,4.59,5.0,69.8,172.7,66,1904,18,4.04,4.03,4.03,16.6,80.5,70,47,3.17,3.17,34.4,29.9
grevlex-16,137,1376,4.99,4.89,4.88,4.88,2.1,85.2,305.2,369,1250,317,4.29,4.29,4.28,14.1,111.7,73,50,3.36,3.36,32.6,52.6
grevlex-17,154,1649,5.27,5.14,5.14,5.13,2.5,82.6,330.8,192,2106,125,4.51,4.50,4.49,14.7,163.2,110,87,3.53,3.53,33.0,83.4
grevlex-18,172,1956,5.67,5.42,5.41,5.40,4.8,70.8,352.8,150,1533,91,4.76,4.76,4.75,16.3,211.6,84,61,3.72,3.72,34.4,125.6
grevlex-19,191,2299,5.94,5.72,5.71,5.71,3.8,75.7,529.0,229,1377,183,5.03,5.03,5.02,15.5,267.5,84,61,3.93,3.93,33.7,208.5
grevlex-20,211,2680,6.06,5.99,5.98,5.98,1.3,90.0,746.1,242,1183,205,5.28,5.26,5.26,13.2,454.4,167,144,4.13,4.13,31.9,333.5
grevlex-21,232,3101,6.53,6.23,6.23,6.23,4.6,70.9,772.6,124,1830,79,5.50,5.50,5.50,15.7,585.1,183,160,4.30,4.3,34.1,485.0
grevlex-22,254,3564,6.83,6.53,6.52,6.51,4.8,70.2,1134.0,178,2240,116,5.75,5.75,5.74,16.0,792.7,188,165,4.49,4.49,34.3,618.4
grevlex-23,277,4071,6.99,6.76,6.74,6.74,3.6,74.5,1310.9,212,1237,176,6.03,6.02,6.02,13.9,956.5,299,276,4.70,4.7,32.7,1099.7
grevlex-24,301,4624,7.35,7.04,7.06,6.70,8.9,45.0,1860.5,544,1558,500,6.28,6.28,6.16,16.3,1836.0,523,500,4.90,4.9,33.4,1601.8
grevlex-25,326,5225,7.72,7.29,7.26,7.28,5.8,69.0,2984.0,463,2317,411,6.50,6.51,6.28,18.7,2544.0,523,500,5.08,5.08,34.3,2117.4
grevlex-26,352,5876,7.94,7.57,7.56,7.51,5.5,73.8,4133.8,557,2373,500,6.75,6.76,6.27,21.0,2169.6,523,500,5.27,5.27,33.7,3393.4
grevlex-27,379,6579,8.13,7.70,7.70,7.64,6.0,69.2,3515.5,537,1450,500,7.02,7.03,6.55,19.4,2747.3,523,500,5.48,5.48,32.6,4405.0
\end{filecontents*}
\csvreader[head to column names, late after line = \\]{results2-grevlex.csv}{n=\size,m=\edges}
		{\instance & \size & \edges & \ubopt & \lbdnn &  \gapdnn & \timednn & \lb & \gap & \time & \ncuts & \iterations}
		\vspace*{-0.25cm} &  &  &  &  &   &  &  &  &  &  & \\
		\begin{filecontents*}{results2-dimacs.csv}
instance,n,m,ubopt,foptlast,btnu,lb,gap,relgapclosed,time,iterations,ncuts,iterationsextra,foptlastdnn,btnudnn,lbdnn,gapdnn,timednn,iterationsdnn,iterationsextradnn,lbdnnsmall,gapdnnsmall,timednnsmall
karate,34,78,0.59,0.59,0.59,0.59,0.0,99.9,5.2,28,518,0,0.55,0.55,0.55,6.7,5.5,23,0,0.24,59.2,0.1
polbooks,105,441,0.37,0.36,0.36,0.35,3.0,80.0,175.6,40,1033,2,0.32,0.32,0.31,14.9,114.2,146,123,0.18,50.0,13.9
football,115,613,1.07,1.07,1.07,1.07,0.0,99.7,118.1,28,382,0,0.98,0.98,0.97,9.4,86.4,28,5,0.76,29.3,21.4
celegansneural,297,2148,1.00,1.00,1.00,0.96,3.6,56.4,3844.4,541,841,500,1.00,1.00,0.92,8.2,2595.9,523,500,0.43,56.7,1366.6
\end{filecontents*}
\csvreader[head to column names, late after line = \\]{results2-dimacs.csv}{n=\size,m=\edges}
		{\instance & \size & \edges & \ubopt & \lbdnn &  \gapdnn & \timednn & \lb & \gap & \time & \ncuts & \iterations}
		\vspace*{-0.25cm} &  &  &  &  &   &  &  &  &  &  & \\
		\begin{filecontents*}{results2-network.csv}
instance,n,m,ubopt,foptlast,btnu,lb,gap,relgapclosed,time,iterations,ncuts,iterationsextra,foptlastdnn,btnudnn,lbdnn,gapdnn,timednn,iterationsdnn,iterationsextradnn,lbdnnsmall,gapdnnsmall,timednnsmall
moviegalaxies-52,59,119,0.54,0.54,0.54,0.54,0.3,98.4,21.5,27,402,0,0.44,0.44,0.43,19.2,21.0,56,33,0.18,65.7,0.9
train-terrorists,64,243,0.60,0.60,0.60,0.60,0.1,92.7,24.6,28,293,0,0.60,0.60,0.59,1.9,46.7,60,37,0.17,72.0,2.2
highschool,70,274,0.91,0.90,0.90,0.90,1.7,93.5,52.4,61,1140,0,0.68,0.68,0.68,25.8,33.4,30,7,0.41,55.7,2.5
blumenau-drug,75,181,0.50,0.50,0.40,0.49,2.5,62.8,71.4,77,785,33,0.50,0.48,0.47,6.8,46.6,61,38,0.19,62.0,2.5
sp-office,92,9827,3.37,3.30,3.30,3.29,2.4,66.8,89.8,50,1333,2,3.14,3.13,3.13,7.1,64.9,42,19,1.72,48.9,7.0
game-thrones,107,352,0.40,0.40,0.40,0.40,0.4,85.9,108.5,32,279,3,0.40,0.40,0.39,2.7,127.3,28,5,0.12,69.2,12.1
revolution,141,160,0.10,0.10,0.10,0.09,7.9,62.3,233.7,34,1110,1,0.08,0.08,0.08,21.0,242.2,396,373,0.04,59.9,45.0
malariagenes-HVR1,307,2812,0.24,0.24,0.24,0.23,4.4,91.5,2620.7,477,1608,446,0.21,0.21,0.11,52.3,3225.4,523,500,0.11,53.6,2242.2
\end{filecontents*}
\csvreader[head to column names, late after line = \\]{results2-network.csv}{n=\size,m=\edges}
		{\instance & \size & \edges & \ubopt & \lbdnn &  \gapdnn & \timednn & \lb & \gap & \time & \ncuts & \iterations}
	\end{longtable}
\endgroup
\end{landscape}

\section{Conclusion}\label{sec:conclusion}
This work demonstrates the effectiveness of the convexification technique by He, Liu, and Tawarmalani~\cite{he2024convexification} in addressing the challenging problem of computing the edge expansion of a graph. By formulating the edge expansion as a completely positive program and introducing a doubly non-negative relaxation enhanced by cutting planes, we have developed a new approach for obtaining strong lower bounds. The augmented Lagrangian algorithm further improves the computational efficiency of solving this relaxation. Additionally, we provide a post processing routine to derive valid lower bounds. Numerical results show that our approach yields tight bounds and performs efficiently even for large graphs. For several instances, the rounded bound is equal to the optimal value or yields a very small relative gap.
Some possible extensions of our approach would be to further strengthen the relaxation by adding McCormick inequalities.

\bibliographystyle{hplain}
\bibliography{references}

\appendix

\end{document}